\documentclass[leqno,11pt]{amsart}
\usepackage{amssymb,amsmath,latexsym,amsfonts,amsbsy,amsthm,mathtools,graphicx,color,subeqnarray,cases,tikz} 
\usepackage{amsfonts}
\usepackage{color}
\usepackage{amssymb,amsthm,amsmath,hyperref}
\usepackage{epsfig}
\usepackage{graphicx,float}
\usepackage{CJK}
\usepackage{setspace}
\usepackage{cancel}
\usepackage{verbatim}
\usepackage{ulem}
\newtheorem{thm}{Theorem}[section]

\newtheorem{lem}[thm]{Lemma}
\newtheorem{rem}[thm]{Remark}
\newtheorem{prop}[thm]{Proposition}

\newcommand{\beq}{\begin{eqnarray}}
\newcommand{\eeq}{\end{eqnarray}}
\newcommand{\beqno}{\begin{eqnarray*}}
\newcommand{\eeqno}{\end{eqnarray*}}
\newcommand{\be}{\begin{equation}}
\newcommand{\ee}{\end{equation}}
\newcommand{\beno}{\begin{equation*}}
\newcommand{\eeno}{\end{equation*}}

\newcommand{\R}{\mathbb{R}}
\newcommand\Z{\mathbb{Z}}

\newcommand\fr{\frac}

\def\e{{\rm e}}
\def\dd{{\rm d}}

\def\ddt{{\frac{\dd}{\dd t}}}
\def\R {\mathbb{R}}

\def \l {\langle}
\def \r {\rangle}
\def\T {{\mathbb T}}
\def\de{{\partial}}

\newcommand\om{\omega}

\newcommand\RR{\mathfrak{R}}

\newcommand{\n}{\notag}

\newcommand{\norm}[1]{\left\lVert#1\right\rVert}

\numberwithin{equation}{section}  % 使 equation 环境的编号随 section 递增
\allowdisplaybreaks[4]

\begin{document}
\title[2D Navier-Stokes Boussinesq equations around the Poiseuille flow]{Stability threshold of the Poiseuille flow in   2D Navier-Stokes Boussinesq equations}

\author{Nangao Zhang }
\address[N. Zhang]{Corresponding author. Research Center for Applied Mathematics and Interdisciplinary Sciences,  School of Mathematics and Statistics, Wuhan Textile University, Wuhan 430200, P. R. China.}
\email{mathngzhang@163.com}

%\author{Gengwei Zhong}
%\address[G. Zhong]{School of Mathematics and Statistics,  Central China Normal University, Wuhan,  430079,  P. R. China.}
%\email{gwzhong@mails.ccnu.edu.cn}

%\author{Ruizhao Zi}
%\address[R. Zi]{School of Mathematics and Statistics, and Key Laboratory of Nonlinear Analysis \& Applications (Ministry of Education), Central China Normal University, Wuhan,  430079,  P. R. China.}
%\email{rzz@ccnu.edu.cn}

\date{\today}

\begin{abstract}
We study the transition threshold problem for the 2D Navier--Stokes Boussinesq equations around the Poiseuille flow on three domains: $\mathbb{T} \times \mathbb{R}$, $\mathbb{T} \times [0, \infty)$, and $\mathbb{T} \times [-1, 1]$. 
More precisely, we prove that if the initial perturbations of the vorticity and the temperature are assumed size $O(\min\{\nu^{\fr34},\mu^{\fr34},\nu^{-\fr14}\mu\})$ and $O(\min\{\nu^{\fr94},\mu^{\fr94}\})$ in an anisotropic Sobolev space respectively, then   for the solution to this system, its size remains so for all time, and the nonlinear enhanced dissipation persists with a rate proportional to $O(\min\{\nu^{\fr12},\mu^{\fr12}\})$.

\end{abstract}	

    
\maketitle

\section{Introduction}

In this paper we consider the two-dimensional Navier-Stokes Boussinesq equations on the domain $\T\times D$, where $D=\R, D=[0, \infty)$, or $D=[-1,1]$:
\begin{equation}\label{1.1}
	\begin{cases}
		\partial_t v+v \cdot \nabla v+\nabla \Pi-\nu \Delta v=\vartheta g e_2, \\
		\partial_t \vartheta+v \cdot \nabla \vartheta-\mu\Delta \vartheta=0, \\
		\operatorname{div} v=0, \ \ \ \ \  (x,y)\in \T\times D,\\ 
		 \de_y v_1(t,x,y\in \de D)=2y,\ v_2(t,x,y\in \de D)=0,\ \vartheta(t,x,y\in \de D)=a,\\
	\end{cases}
\end{equation}
where $v=\left(v^1(t, x, y), v^2(t, x, y)\right)^T$is the velocity field, $\vartheta=\vartheta(t, x, y)$ is the temperature, and $\Pi=\Pi(t, x, y)$ is the pressure. 
Moreover, $g$ is the gravitational constant and $e_2=(0,1)^{T}$ is the unit vector in the vertical direction.	
The positive  constants $\nu$ and $\mu$ represent the viscosity coefficient and the thermal diffusivity, respectively; 
$a$ is an arbitrary constant.
 Without loss of generality, we normalize the constants such that $ g = 1$ for simplicity.

Defining the vorticity of $v$ as $\Omega :=\nabla^\bot \cdot v$, where $\nabla^\perp=(-\de_y,\de_x)$,
 it is possible to remove the pressure term and to rewrite the above system as
\begin{align}\label{1.2}
\begin{cases}
\de_t \Omega+v\cdot \nabla \Omega-\nu \Delta\Omega=\de_x\vartheta,\\
\partial_t \vartheta+v \cdot \nabla \vartheta-\mu \Delta \vartheta=0, \\
v=\nabla^\perp \Phi, \quad \Delta\Phi=\Omega,\ \ \ \ \  (x,y)\in \T\times D,\\ 
\Omega(t,x,y\in \de D)=-2y,\ \ \vartheta(t,x,y\in \de D)=a,
\end{cases}
\end{align}
where $\Phi(t,x,y)$ is the corresponding stream-function.

Obviously, system \eqref{1.1} admits a class of steady-states, where the velocity field is the \emph{Poiseuille flow}, the pressure is a linear function of the vertical component and the temperature is a constant, namely
\begin{equation}\notag
v_{s}(x,y)=(y^2,0), \ \ \ \Pi_{s}(x,y)=2\nu x+ay+b, \ \ \ \vartheta_s=a,	
\end{equation}
here $a,b \in \mathbb{R}$ are arbitrary constants.
The main goal of this article is to prove quantitative stability results for the Poiseuille flow. 
 Let us define the perturbations as $u=v-(y^2,0)$ and $\theta=\vartheta-a$, with corresponding $\omega=\nabla^\bot \cdot u=\Omega+2y$, then one can rewrite \eqref{1.2} as
\begin{align}\label{1.3}
\begin{cases}
\de_t \omega+y^2\de_x\omega -2\de_x\phi-\nu\Delta \omega+u\cdot \nabla \omega=\de_x\theta,\\
\partial_t \theta+y^2\de_x\theta+u \cdot \nabla \theta-\mu \Delta \theta=0, \\ 
u=\nabla^\perp \phi, \quad \Delta\phi=\omega,\\ 
\omega(0,x,y)=\omega_{in}(x,y),\ \ \ \theta(0,x,y)=\theta_{in}(x,y), \ \  (x,y)\in \T\times D,\\ 
\omega(t,x,y\in \de D)=0,\ \ \theta(t,x,y\in \de D)=0,\\ 
\end{cases}
\end{align}
where $u$, $\omega$ and $\phi$ are the perturbations of the velocity, vorticity and stream-function near the Poiseuille flow. 

In this paper, we will study the following stability threshold problem.

{\it Given norms $\|\cdot\|_{Y_1}$ and $\|\cdot\|_{Y_2}$, find $\alpha=\alpha\left(Y_1, Y_2\right)$ and $\beta=\beta\left(Y_1, Y_2\right)$ such that

$$
\begin{aligned}
	& \left\|u_{\text {in}}\right\|_{Y_1} \leq \nu^\alpha \text { and }\left\|\theta_{\text {in}}\right\|_{Y_2} \leq \nu^\beta \Rightarrow \text { stability; } \\
	& \left\|u_{\text {in}}\right\|_{Y_1} \gg \nu^\alpha \text { or }\left\|\theta_{\text {in}}\right\|_{Y_2} \gg \nu^\beta \Rightarrow \text { instability. }
\end{aligned}
$$}
The exponent $\alpha $ and $\beta$ is referred to as the transition threshold.

\subsection{Background and previous works}
The Navier-Stokes Boussinesq (NSB) equations describes the dynamic behavior of incompressible and non-uniform fluid flows driven by buoyancy. This system is typically used to simulate atmospheric sciences, oceanic circulation, mantle convection, and engineering heat transfer (cf. \cite{Gill1982,Pedlosky1987}), as well as to model Rayleigh-B\'enard convection in laminar and turbulent flow states, for example, see \cite{Doering1996,Getling1998}. 
From a mathematical perspective, the two-dimensional inviscid NSB system can be related to the three-dimensional Euler equation for axisymmetric rotating flows away from the singularity. 
Moreover, the two-dimensional NSB system retains some key features of the three-dimensional Euler and Navier-Stokes equations, such as the vortex stretching mechanism \cite{Majda2002}.

Due to its wide application in physics and its significance in mathematics, the theory of the NSB equations has received extensive attention.
According to different regions, the research on this equation can be roughly divided into the following categories:

For the region $\mathbb{R} \times \mathbb{R}$, when considering perturbations around the constant equilibrium states $v_s=(0,0), \theta_s=1$, Chae \cite{Chae} obtained
 the global in time regularity with either the zero diffusivity $\mu=0$ or the zero viscosity $\nu=0$, Cao and Wu \cite{Cao2013} proved the global regularity of this system. 
 For the hydrostatic equilibrium states $v_s=(0,0), \vartheta_s=\vartheta_s(y)$, Doering et al. \cite{Doering2018} studied the long time behavior of solution to this system without buoyancy diffusion.
 For other related results, refer to \cite{Danchin,Li2016,Tao2020} and references therein.

On the region $\mathbb{T} \times \mathbb{R}$, for the perturbations around the Couette flow without hydrostatic equilibrium, i.e., $v_s = (y, 0)$, $\vartheta_s = 1$, Tao and Wu~\cite{Tao2019} obtained the linear stability of system ~\eqref{1.1}. 
 Subsequently, by constructing some Fourier multipliers, 
 Deng et al.~\cite{Deng2020} proved that for viscosity coefficient $\nu=\mu>0$, if the initial perturbation satisfies that for $b>1$,
\[
\|\omega_{in}\|_{H^b} \ll \nu^{\frac{1}{2}},\quad \|\theta_{in}\|_{H^b} \ll \nu,\quad \||D_x|^{\frac{1}{3}}\theta_{in}\|_{H^b} \ll \nu^{\frac{5}{6}},
\]
the Boussinesq system with full dissipation is nonlinear stability near Couette flow. 
For the case with only vertical dissipation, they also obtained the nonlinear stability results when
\[
\|\omega_{in}\|_{H^b} \ll \nu^{\frac{2}{3}},\quad \|\theta_{in}\|_{H^b} \ll \nu,\quad \||D_x|^{\frac{1}{3}}\theta_{in}\|_{H^b} \ll \nu ,
\] 
with $b>\frac{4}{3}$.
Later, Zhang and Zi \cite{ZhangZi} obtained the nonlinear stability of Couette flow for the Boussinesq equations with vertical dissipation if the initial data satisfies
\[
\|u_{in}\|_{H^{b+1}} + \nu^{-\frac{1}{2}}\|\theta_{in}\|_{H^b} + \nu^{-\frac{1}{3}}\||D_x|^{\frac{1}{3}} \theta_{in}\|_{H^b} \ll \nu^{\frac{1}{3}}, 
\]
with $b > 7$, thereby improving the results in \cite{Deng2020}.
Recently, by using the idea of quasi-linearization, Niu and Zhao \cite{Niu} obtained the nonlinear stability of Couette flow for the Boussinesq equations with full dissipation when the initial data satisfies
\[
\|u_{in}\|_{H^{b+1}} + \nu^{-\frac{1}{3}}\|\theta_{in}\|_{H^b} + \nu^{-\frac{1}{3}}\|\left \langle D_x \right \rangle  \theta_{in}\|_{H^b}\ll \nu^{\frac{1}{3}}, \quad b > 5.
\]
If the equilibrium states are Couette flow with hydrostatic equilibrium, i.e., $v_s = (y, 0)$, $\vartheta_s = \vartheta_s(y)$, Yang and Lin~\cite{yang2019} proved the linear stability of system~\eqref{1.1} in the inviscid case. 
Bedrossian et al.~\cite{Bedrossian20231} established the nonlinear shear-buoyancy instability until the nonlinear timescale $t \approx \varepsilon^{-2}$. 
More recently, Zhai et al.~\cite{zhai2022} studied the nonlinear asymptotic stability of the Couette flow in the stably stratified regime.
For other results, refer to \cite{MSZ} and the references therein.
 
For the region $\mathbb{T} \times [-1, 1]$, the research on stability is still relatively limited. 
When the steady state solution is $v_s = (y, 0)$ and $\vartheta_s = 1$, by using the resolvent estimation, Masmoudi et al.~\cite{Masmoudi2021} proved that if the initial perturbation satisfies
\begin{equation}\notag
\|u_{in}\|_{H^s} \ll \nu^{\frac{1}{2}}, \quad \| \theta_{in} \|_{H^1} + \left\| \langle \partial_x \rangle^{\frac{1}{2}} \theta_{in} \right\|_{H^1} \ll \nu^{\frac{11}{12}},
\end{equation}
then the system is asymptotically stable.
When the steady state solution is near-Couette shear flows, namely $v_s = (U(t,y), 0),~\vartheta_s = 1$, where $U(t,y) = y + \partial_y \int_{-1}^{1} G(y,y') W(t,y'){\rm d}y'$, Liang et al. \cite{Liang2025} recently studied the nonlinear enhanced dissipation effect and the inviscid damping effect. 
When the steady state solution is the Poiseuille flow, i.e. $v_s = (1-y^2, 0)$ and $\vartheta_s = 0$. By the resolvent estimate method, Wang \cite{Wang2025}  proved that if the initial velocity and initial temperature satisfies
\[
\| u_{in} - (1 - y^2, 0) \|_{H^{\frac{7}{2}+}} \ll \min\left\{ \mu, \nu \right\}^{\frac{2}{3}},\  \| \theta_{in} \|_{H^1} + \| |D_x|^{\frac{1}{8}} \theta_{in} \|_{H^1} \ll \min\left\{ \mu, \nu \right\}^{\frac{31}{24}},
\]
then the system is nonlinear stability. Actually, Wang \cite{Wang2025} only got the stability result but not the enhanced dissipation estimate on $\mathbb{T} \times [-1, 1]$.
For other results related to \eqref{1.1}, refer to \cite{Bianchini2025,Knobel}.

Based on the above observations, it is not difficult to find that considerable progress has been made in the research on the stability of NSB equations near the constant state or Couette flow, but the research on the stability of Poiseuille flow is very few, especially on the whole plane $\mathbb{T} \times \mathbb{R}$ or the half plane $\mathbb{T} \times [0, \infty)$. 
In the especial case where $\vartheta=0$,  \eqref{1.1} reduces to 2D incompressible Navier-Stokes system. 
By using  the semigroup estimate,  Coti Zelati et al. \cite{Coti2020} first gave the stability and enhanced dissipation estimate of 2D incompressible Navier-Stokes equations near the Poiseuille flow on $\mathbb{T} \times \mathbb{R}$ and the transition threshold $\alpha\leq \frac{3}{4}+$. 
 This result has been further improved to $\alpha\leq \frac{2}{3}+$ by Del Zotto \cite{Del2023}.
Also about Navier-Stokes equations, Ding et al. \cite{Ding2025}  investigated the case with Navier-slip boundary in a finite channel $\mathbb{T} \times [-1, 1]$, and proved that the transition threshold is $\alpha\leq \frac{2}{3}$ by the resolvent estimate method. 
For other results, see \cite{Arbon2025,BPM1,BPM2,BMV,BVW,Wang2025a,Chen2025,Chen2023}.
Inspired by the above results, in this paper, we will consider the transition threshold problem for the 2D NSB equations around the Poiseuille flow on three unbounded domains: the whole plane $\mathbb{T} \times \mathbb{R}$, the half plane $\mathbb{T} \times [0, \infty)$, and the infinite channel $\mathbb{T} \times [-1, 1]$. 
To be precise, we will introduce some energy functionals based on the hypocoercivity method \cite{Villani}, then by energy method, we can obatin the enhanced dissipation rate of the system.

{\bf Notations and conventions.}
We introduce some notation used throughout the paper.
For $a, b \in \mathbb{R}$, we denote
\begin{equation}\notag
 \langle a\rangle=\left(1+|a|^2\right)^{\frac{1}{2}}, \ \ \ |a, b|=|a|+|b|.
 \end{equation} 
We use the notation $a \lesssim b$  to indicate that  there exists a constant $C>0$ (independent of $\nu, \mu$) such that $a \leq C b$ (we define $a \gtrsim b$ analogously). Similarly, we use the notation $a \approx b$ when there exists $C>0$ such that $C^{-1} b \leq a \leq C b$.

For a function $f=f(t,x, y): \mathbb{T} \times \mathbb{R} \rightarrow \mathbb{R}$ and $k\in \mathbb{Z}$, we denote 
\begin{align}\notag
f_k(t,y)=\frac{1}{2\pi}\int_\T f(t,x,y)\e^{-ikx}\dd x,\ \ \ \ \ \ f_{\neq}=f-f_0.
\end{align}

For $D=\mathbb{R},[0, \infty)$, or $[-1,1]$, we use the notation $\|\cdot\|_{L^p}$ to denote the $L^p(D)$ norm, and $\|\cdot\|_{L^2}=\|\cdot\|$. The notation $\langle\cdot, \cdot\rangle$ will denote the inner product on $L^2(D)$.

\subsection{ Main results}
We firstly consider the case that  $(x,y)\in \mathbb{T} \times \mathbb{R}$ or $(x,y)\in\mathbb{T} \times [0, \infty)$.
\begin{thm}\label{thm1.1}
Suppose that $(\om_{in},\theta_{in})$ is the initial data for \eqref{1.3} with $D=\mathbb{R}$ or $[0,\infty)$.  For all $m>\fr12$, there exists a constant $\varepsilon_0>0$ independent of $\nu$ and $\mu$ such that if
\begin{equation}\label{1.4}
	\|u_{0,in}^1\|+\left\|y\de_y\om_{0,in}\right\|+\sum_{j=0}^1\left\|(\nu^{\fr14}\de_y)^j\langle\de_x\rangle^{m-\frac{j}{4}}\om_{in}\right\|+\nu^{-\fr14}\left\||\de_x|^{m+\frac{1}{4}}(y\om_{in},\nabla \phi_{in})\right\|\leq \varepsilon_0\min\{\nu^{\fr34},\mu^{\fr34},\nu^{-\fr14}\mu\},
\end{equation}
\begin{equation}\label{1.5}
	\left\|y\de_y\theta_{0,in}\right\|+\sum_{j=0}^1\left\|(\mu^{\fr14}\de_y)^j\langle\de_x\rangle^{m-\frac{j}{4}}\theta_{in}\right\|+\mu^{-\fr14}\left\||\de_x|^{m+\frac{1}{4}}y\theta_{in}\right\|\leq \varepsilon_0\min\{\nu^{\fr94},\mu^{\fr94}\},
\end{equation}
then for $0<\delta_0\ll 1$ (depending only on universal constants), $t\geq 0$ and $0<\nu,\mu\leq 1$,  the corresponding solution $(\omega,\theta)$ to \eqref{1.3} satisfies
\begin{equation}\label{1.6}
\sum_{j=0}^1\left\|(\nu^{\fr14}\de_y)^j\langle\de_x\rangle^{m-\frac{j}{4}}\om_{\neq}\right\|+\nu^{-\fr14}\left\||\de_x|^{m+\frac{1}{4}}(y\om_{\neq},\nabla \phi_{\neq})\right\|\lesssim \varepsilon_0{\rm e}^{-\delta_0\min\{\nu,\mu\}^{\frac12} t},	
\end{equation}
\begin{equation}\label{1.7}
	\sum_{j=0}^1\left\|(\mu^{\fr14}\de_y)^j\langle\de_x\rangle^{m-\frac{j}{4}}\theta_{\neq}\right\|+\mu^{-\fr14}\left\||\de_x|^{m+\frac{1}{4}}y\theta_{\neq}\right\|\lesssim \varepsilon_0{\rm e}^{-\delta_0\min\{\nu,\mu\}^{\frac12} t}.
\end{equation}
\end{thm}

For the case of the finite channel, since the Poincar\'e inequality can be applied, we can also obtain estimates of the solution at zero frequency. The specific statement of the theorem is as follows:
\begin{thm}\label{thm1.2}
Suppose that $(\om_{in},\theta_{in})$ is the initial data for \eqref{1.3} with $D=[-1,1]$.  For all $m>\fr12$, there exists a constant $\varepsilon_0>0$ independent of $\nu$ and $\mu$ such that if
\begin{equation}\label{1.8}
	\|u_{0,in}^1\|+\left\|y\de_y\om_{0,in}\right\|+\sum_{j=0}^1\left\|(\nu^{\fr14}\de_y)^j\langle\de_x\rangle^{m-\frac{j}{4}}\om_{in}\right\|+\nu^{-\fr14}\left\||\de_x|^{m+\frac{1}{4}}(y\om_{in},\nabla \phi_{in})\right\|\leq \varepsilon_0\min\{\nu^{\fr34},\mu^{\fr34},\nu^{-\fr14}\mu\},
\end{equation}
\begin{equation}\label{1.9}
	\left\|y\de_y\theta_{0,in}\right\|+\sum_{j=0}^1\left\|(\mu^{\fr14}\de_y)^j\langle\de_x\rangle^{m-\frac{j}{4}}\theta_{in}\right\|+\mu^{-\fr14}\left\||\de_x|^{m+\frac{1}{4}}y\theta_{in}\right\|\leq \varepsilon_0\min\{\nu^{\fr94},\mu^{\fr94}\},
\end{equation}
then for $0<\delta_0\ll 1$  (depending only on universal constants), $t\geq 0$ and $0<\nu,\mu\leq 1$,  the corresponding solution $(\omega,\theta)$ to \eqref{1.3} satisfies
\begin{equation}\label{1.10}
\sum_{j=0}^1\left\|(\nu^{\fr14}\de_y)^j\om_{0}\right\|+\left\|y\de_y\om_{0}\right\|\lesssim \varepsilon_0{\rm e}^{-\delta_0\min\{\nu,\mu\} t},	
\end{equation}
\begin{equation}\label{1.11}
	\sum_{j=0}^1\left\|(\mu^{\fr14}\de_y)^j\theta_{0}\right\|+\left\|y\de_y\theta_{0}\right\|\lesssim \varepsilon_0{\rm e}^{-\delta_0\min\{\nu,\mu\} t},
\end{equation}
\begin{equation}\label{1.12}
\sum_{j=0}^1\left\|(\nu^{\fr14}\de_y)^j\langle\de_x\rangle^{m-\frac{j}{4}}\om_{\neq}\right\|+\nu^{-\fr14}\left\||\de_x|^{m+\frac{1}{4}}(y\om_{\neq},\nabla \phi_{\neq})\right\|\lesssim \varepsilon_0{\rm e}^{-\delta_0\min\{\nu,\mu\}^{\frac12} t},	
\end{equation}
\begin{equation}\label{1.13}
	\sum_{j=0}^1\left\|(\mu^{\fr14}\de_y)^j\langle\de_x\rangle^{m-\frac{j}{4}}\theta_{\neq}\right\|+\mu^{-\fr14}\left\||\de_x|^{m+\frac{1}{4}}y\theta_{\neq}\right\|\lesssim \varepsilon_0{\rm e}^{-\delta_0\min\{\nu,\mu\}^{\frac12} t}.
\end{equation}
\end{thm}
\begin{rem}
 According to the results about Navier-Stokes equations \cite{Coti2020,Del2023}, the authors used energy theory and semigroup estimation to get the transition threshold. 
  Because of the coupling effect of this system, it seems difficult to use semigroup estimation like Navier-Stokes equations.
  In this paper, we use a new energy estimate to study the enhanced dissipation of NSB equations, which will provide some reference for studying the stability of Poiseuille flow for other coupled models. 
\end{rem}
\begin{rem}
     Our main results, the Theorems \ref{thm1.1} and \ref{thm1.2}, give the enhanced dissipation with decay rate ${\rm e}^{-\delta_0\min\{\nu,\mu\}^{\frac12} t}$,  which is consistent with the Navier-Stokes equations \cite{Del2023,Ding2025}. However, the transition threshold may not be sharp.  It is very interesting and challenging to study the optimal transition threshold of
 the Poiseuille flow in the future.
\end{rem}
\begin{rem}
     If $v$ is non-slip boundary conditions, the boundary layers will necessitate a more complicated approach that cannot be done only with energy estimates, we will continue to consider it in the future.
\end{rem}

\section{Outline}

Due to the coupling effect of vorticity and temperature, the first challenge in proving Theorems \ref{thm1.1} and \ref{thm1.2} is to obtain sufficiently good estimates for the linearized problem of the system. 
Secondly, for the nonlinear terms, we will handle them through the smallness of the initial values and the dissipation obtained from the analysis of the linearized system. One of the main contributions of this paper is to provide a more direct method based on energy estimation.
In the following, we will provide a detailed explanation of the proof process of the Theorems \ref{thm1.1} and \ref{thm1.2}. 
First, by performing Fourier transform on the system \eqref{1.3} with respect to $x$,  we get the zero mode \eqref{zero} and nonzero mode equations \eqref{4.1} respectively.
Next, we set $\kappa:=\min\{\nu,\mu\}$, and define the following functionals.
When $k=0$, 
\begin{equation}\label{2.2}
	\mathcal{E}_{\omega,0}:=\fr14\left[2\|\omega_0\|^2+\|y\de_y\omega_0\|^2+2\nu^{\fr12}\|\de_y\omega_0\|^2\right],
	\end{equation}
\begin{equation}\label{2.3}
	\mathcal{E}_{\theta,0}:=\fr14\left[2\|\theta_0\|^2+\|y\de_y\theta_0\|^2+2\mu^{\fr12}\|\de_y\theta_0\|^2\right],
	\end{equation}
\begin{equation}\label{2.4}
	\mathcal{D}_{\omega,0}:=\nu\|\de_y\omega_0\|^2+\nu\|y\de_{yy}\omega_0\|^2+\nu^{\fr32}\|\de_{yy}\omega_0\|^2+\|\de_y u_0\|^2,
	\end{equation}
\begin{equation}\label{2.5}
	\mathcal{D}_{\theta,0}:=\mu\|\de_y\theta_0\|^2+\mu\|y\de_{yy}\theta_0\|^2+\mu^{\fr32}\|\de_{yy}\theta_0\|^2.
	\end{equation}
When $k\neq 0$, we define
\begin{align}\label{2.6}
E_{\omega}[\omega_k]:=\frac12&\left[\|\omega_k\|^2 +\nu^{\frac12}|k|^{-\frac12} \|\de_y\omega_k\|^2+5|k|^{-1} \RR\l iky\omega_k,\de_y\omega_k\r \right. \n\\
		& \left.+100\nu^{-\frac12}|k|^{\frac12}\left(\|y\omega_k\|^2+ 2 \|\nabla_k\phi_k\|^2\right) \right],	
\end{align}
\begin{align}\label{2.7}
	E_{\theta}[\theta_k]:=\frac12&\left[\|\theta_k\|^2 +2^{-7}\mu^{\frac12}|k|^{-\frac12} \|\de_y\theta_k\|^2+2^{-10}|k|^{-1} \RR\l iky\theta_k,\de_y\theta_k\r \right. \n\\
		& \left.+2^{-9}\mu^{-\frac12}|k|^{\frac12}\|y\theta_k\|^2\right],
		\end{align}
		\begin{align}\label{2.8}
	D_{\omega}[\omega_k]:=&\nu^{\frac12}|k|^{\frac12}\norm{\omega_k}^2 +\nu\|\nabla_k\omega_k\|^2+\nu^{\frac32}|k|^{-\frac12}\|\de_y\nabla_k\omega_k\|^2+|k|\| y\omega_k\|^2\n\\ 
&+|k|\|\nabla_k\phi_k\|^2+\nu^{\frac12}|k|^{\frac12}\|y\nabla_k\omega_k\|^2 ,		
		\end{align}
	\begin{align}\label{2.9}
		D_{\theta}[\theta_k]:=&\mu^{\frac12}|k|^{\frac12}\norm{\theta_k}^2+\mu\|\nabla_k\theta_k\|^2+\mu^{\frac32}|k|^{-\frac12}\|\de_y\nabla_k\theta_k\|^2+|k|\| y\theta_k\|^2\n\\ 
		&+\mu^{\frac12}|k|^{\frac12}\|y\nabla_k\theta_k\|^2,
	\end{align}
	\begin{equation}\label{2.10}
	\mathcal{E}_{\om,\neq}:={\rm e}^{2\delta_0\kappa^{\frac{1}{2}} t}\sum_{k\in \Z\backslash\{0\}}|k|^{2m}E_{\om}[\om_k],
		\end{equation}
	\begin{equation}\label{2.11}
	\mathcal{E}_{\theta,\neq}:={\rm e}^{2\delta_0\kappa^{\frac{1}{2}} t}\sum_{k\in \Z\backslash\{0\}}|k|^{2m}E_{\theta}[\theta_k],
		\end{equation}
\begin{equation}\label{2.12}
	\mathcal{D}_{\om,\neq}:={\rm e}^{2\delta_0\kappa^{\frac{1}{2}} t}\sum_{k\in \Z\backslash\{0\}}|k|^{2m}D_{\om}[\om_k],
		\end{equation}
	\begin{equation}\label{2.13}
	\mathcal{D}_{\theta,\neq}:={\rm e}^{2\delta_0\kappa^{\frac{1}{2}} t}\sum_{k\in \Z\backslash\{0\}}|k|^{2m}D_{\theta}[\theta_k].
		\end{equation}

With the above energy functionals in hand, we are ready to prove the following two results. 
The first result is the estimation of zero mode, which will be proved in Section \ref{s3}.

\begin{prop}\label{p1}
Under the conditions of Theorem \ref{thm1.1}, we have
	\begin{equation}\label{2.13a}
	\ddt\|u_0^1\|^2+\nu\|\de_yu_0^1\|^2\lesssim	\mathcal{E}_{\omega,0}^{\fr12}\mathcal{D}_{\omega,\neq},
	\end{equation}
	\begin{equation}\label{2.14}
\ddt\mathcal{E}_{\omega,0}+\fr12\mathcal{D}_{\omega,0}\lesssim \nu^{-\fr12}\mathcal{E}_{\omega,\neq}^{\fr12}\mathcal{D}_{\omega,0}^{\fr12}\mathcal{D}_{\omega,\neq}^{\fr12},		
	\end{equation}
\begin{equation}\label{2.15}
\ddt\mathcal{E}_{\theta,0}+\frac{1}{2}\mathcal{D}_{\theta,0}\lesssim (\mu^{-\fr12}+\nu^{-\fr14}\mu^{-\fr14})\mathcal{E}_{\theta,\neq}^{\fr12}\mathcal{D}_{\theta,0}^{\fr12}\mathcal{D}_{\omega,\neq}^{\fr12}+\nu^{\fr14}\mu^{-\fr34}\mathcal{E}_{\om,\neq}^{\fr12}\mathcal{D}_{\theta,0}^{\fr12}\mathcal{D}_{\theta,\neq}^{\fr12}.		
	\end{equation}
\end{prop}
The following result is the estimate of non-zero mode, which will be proved in Section \ref{s4}.
\begin{prop}\label{p2}
Under the conditions of Theorem \ref{thm1.1}, there holds
\begin{align}\label{2.16}
\ddt \mathcal{E}_{\omega,\neq} +\fr14\mathcal{D}_{\omega,\neq}
\lesssim&\nu^{-\frac{1}{2}} \mathcal{E}^{\frac{1}{2}}_{\omega,\neq} \mathcal{D}^{\frac{1}{2}}_{\omega, \neq}\mathcal{D}^{\frac{1}{2}}_{\omega, 0}+\nu^{-\frac{3}{4}}\mathcal{E}_{\omega,\neq}^{\frac{1}{2}}  \mathcal{D}_{\omega, \neq}+\nu^{-\frac12}\mathcal{E}_{\omega,0}^{\fr12}\mathcal{D}_{\omega,\neq}+\nu^{-\frac{3}{4}}\|u^1_{0}\|_{L^\infty}\mathcal{D}_{\omega,\neq}\n\\ 
&+(\nu^{-\frac12}\mu^{-\frac56}+\nu^{-1}\mu^{-\fr12})\mathcal{D}_{\theta,\neq},
\end{align}
\begin{align}\label{2.17}
\ddt \mathcal{E}_{\theta,\neq} +2^{-14}\mathcal{D}_{\theta,\neq}
\leq&\mu^{-\frac{1}{2}} \mathcal{E}^{\frac{1}{2}}_{\theta,\neq} \mathcal{D}^{\frac{1}{2}}_{\omega, \neq}\mathcal{D}^{\frac{1}{2}}_{\theta, 0}+(\nu^{-\frac{1}{2}}+\mu^{-\frac{1}{2}}+\nu^{-\frac{1}{4}} \mu^{-\frac{1}{2}})\mathcal{E}_{\theta,\neq}^{\frac{1}{2}}  \mathcal{D}_{\omega, \neq}^{\frac{1}{2}}\mathcal{D}_{\theta, \neq}^{\frac{1}{2}}\n\\ 
&+\nu^{-\frac18}\mu^{-\frac13}\mathcal{E}_{\om,0}^{\fr12}\mathcal{D}_{\theta,\neq}+\mu^{-\frac12}\mathcal{E}_{\theta,0}^{\fr12}\mathcal{D}_{\theta,\neq}^{\fr12}\mathcal{D}_{\om,\neq}^{\fr12}+\mu^{-\frac12}\mathcal{E}_{\theta,\neq}^{\frac12}\mathcal{D}_{\theta,\neq}\n\\  
&+\mu^{-\frac{2}{3}}\|u^1_{0}\|_{L^\infty}\mathcal{D}_{\theta,\neq}+\nu^{\frac{1}{4}}\mu^{-\frac{3}{4}} \mathcal{E}^{\frac{1}{2}}_{\om,\neq} \mathcal{D}^{\frac{1}{2}}_{\theta, \neq}\mathcal{D}^{\frac{1}{2}}_{\theta, 0}\n\\ 
	&+(\nu^{\fr14}\mu^{-\fr34}+\mu^{-\frac{2}{3}}+\nu^{\fr14}\mu^{-1})\mathcal{E}_{\om,\neq}^{\fr12}\mathcal{D}_{\theta,\neq}.
\end{align}
	\end{prop}
Combining the above two results, we now begin to prove Theorem \ref{thm1.1}.

{\bf Proof of Theorem \ref{thm1.1}:} Firstly, for any $k \in \mathbb{Z}$, we define the following two energy functionals:
\begin{equation}\label{2.18}
	\mathfrak{E}_{\om}:=\mathcal{E}_{\omega,0}+\mathcal{E}_{\om,\neq}+\fr14\int_0^t(\mathcal{D}_{\om,0}+\mathcal{D}_{\om,\neq})(\tau){\rm d}\tau,
\end{equation}
\begin{equation}\label{2.19}
	\mathfrak{E}_{\theta}:=\mathcal{E}_{\theta,0}+\mathcal{E}_{\theta,\neq}+\frac{1}{2^{14}}\int_0^t(\mathcal{D}_{\theta,0}+\mathcal{D}_{\theta,\neq})(\tau){\rm d}\tau.
\end{equation} 
Then by combining Propositions \ref{p1}--\ref{p2}, we can conclude  
\begin{align}\notag
	\|u_0^1(t)\|^2+\nu\int_0^t\|\de_yu_0^1(\tau)\|^2{\rm d}\tau\lesssim&	\|u_{0,in}^1\|^2+\left(\sup_{0\leq\tau\leq t}\mathcal{E}_{\omega,0}^{\fr12}(\tau)\right)\int_0^t\mathcal{D}_{\omega,\neq}(\tau){\rm d}\tau\n\\ 
	\lesssim&\|u_{0,in}^1\|^2+\mathfrak{E}_{\omega}(t)^{\fr32},\notag
	\end{align}
     and then
\begin{equation}\label{2.20}
	\|u_0^1(t)\|_{L^\infty}\lesssim\|u_0^1(t)\|^{\fr12}\|\om_0(t)\|^{\fr12}\lesssim\|u_{0,in}^1\|^{\fr12}\mathfrak{E}_{\omega}(t)^{\fr14}+\mathfrak{E}_{\omega}(t)^{\fr58}.
	\end{equation}
In addition, 
\begin{align}\label{2.21}
\mathfrak{E}_{\omega}(t) 
\lesssim& \nu^{-\fr12}\left(\sup_{0\leq\tau\leq t}\mathcal{E}_{\omega,\neq}^{\fr12}(\tau)\right)\int_0^t(\mathcal{D}_{\omega,0}^{\fr12}\mathcal{D}_{\omega,\neq}^{\fr12})(\tau){\rm d}\tau+\nu^{-\frac{3}{4}}\left(\sup_{0\leq\tau\leq t}\mathcal{E}_{\omega,\neq}^{\fr12}(\tau)\right)\int_0^t\mathcal{D}_{\omega, \neq}(\tau){\rm d}\tau\n\\	
&+\nu^{-\frac12}\left(\sup_{0\leq\tau\leq t}\mathcal{E}_{\omega,0}^{\fr12}(\tau)\right)\int_0^t\mathcal{D}_{\omega, \neq}(\tau){\rm d}\tau+\nu^{-\frac{3}{4}}\left(\sup_{0\leq\tau\leq t}\|u^1_{0}(\tau)\|_{L^\infty}\right)\int_0^t\mathcal{D}_{\omega, \neq}(\tau){\rm d}\tau\n\\
&+(\nu^{-\frac12}\mu^{-\frac56}+\nu^{-1}\mu^{-\fr12})\int_0^t\mathcal{D}_{\theta,\neq}(\tau){\rm d}\tau+\mathfrak{E}_{\omega}(0) \n\\ 
\lesssim&\mathfrak{E}_{\omega}(0) +\nu^{-\frac{3}{4}}\left(\sup_{0\leq\tau\leq t}\mathfrak{E}_{\omega}(\tau) \right)^{\frac32}+\nu^{-\frac{3}{4}}\|u_{0,in}^1\|^{\fr12}\left(\sup_{0\leq\tau\leq t}\mathfrak{E}_{\omega}(\tau) \right)^{\frac{5}{4}}\n\\ 
&+(\nu^{-\frac12}\mu^{-\frac56}+\nu^{-1}\mu^{-\fr12})\mathfrak{E}_{\theta}(t),
\end{align}
\begin{align}\label{2.22}
\mathfrak{E}_{\theta}(t) 
\lesssim& (\mu^{-\fr12}+\nu^{-\fr14}\mu^{-\fr14})\left(\sup_{0\leq\tau\leq t}\mathcal{E}_{\theta,\neq}^{\fr12}(\tau)\right)\int_0^t(\mathcal{D}_{\theta,0}^{\fr12}\mathcal{D}_{\omega,\neq}^{\fr12})(\tau){\rm d}\tau\n\\ 
&+\nu^{\fr14}\mu^{-\fr34}\left(\sup_{0\leq\tau\leq t}\mathcal{E}_{\omega,\neq}^{\fr12}(\tau)\right)\int_0^t(\mathcal{D}_{\theta,0}^{\fr12}\mathcal{D}_{\theta,\neq}^{\fr12})(\tau){\rm d}\tau\n\\ 
&+(\nu^{-\frac{1}{2}}+\mu^{-\frac{1}{2}}+\nu^{-\frac{1}{4}} \mu^{-\frac{1}{2}})\left(\sup_{0\leq\tau\leq t}\mathcal{E}_{\theta,\neq}^{\fr12}(\tau)\right)\int_0^t(\mathcal{D}_{\omega, \neq}^{\frac{1}{2}}\mathcal{D}_{\theta, \neq}^{\frac{1}{2}})(\tau){\rm d}\tau\n\\ 
&+\nu^{-\frac18}\mu^{-\frac13}\left(\sup_{0\leq\tau\leq t}\mathcal{E}_{\om,0}^{\fr12}(\tau)\right)\int_0^t\mathcal{D}_{\theta,\neq}(\tau){\rm d}\tau+\mu^{-\frac12}\left(\sup_{0\leq\tau\leq t}\mathcal{E}_{\theta,0}^{\fr12}(\tau)\right)\int_0^t(\mathcal{D}_{\theta,\neq}^{\fr12}\mathcal{D}_{\om,\neq}^{\fr12})(\tau){\rm d}\tau\n\\  
&+\mu^{-\frac12}\left(\sup_{0\leq\tau\leq t}\mathcal{E}_{\theta,\neq}^{\fr12}(\tau)\right)\int_0^t\mathcal{D}_{\theta,\neq}(\tau){\rm d}\tau+\mu^{-\frac{2}{3}}\left(\sup_{0\leq\tau\leq t}\|u^1_{0}(\tau)\|_{L^\infty}\right)\int_0^t\mathcal{D}_{\theta, \neq}(\tau){\rm d}\tau\n\\ 
	&+(\nu^{\fr14}\mu^{-\fr34}+\mu^{-\frac{2}{3}}+\nu^{\fr14}\mu^{-1})\left(\sup_{0\leq\tau\leq t}\mathcal{E}_{\om,\neq}^{\fr12}(\tau)\right)\int_0^t\mathcal{D}_{\theta,\neq}(\tau){\rm d}\tau\n\\ 
	\lesssim&(\nu^{-\frac{1}{2}}+\mu^{-\frac{2}{3}}+\nu^{-\frac{1}{4}} \mu^{-\frac{1}{2}}+\nu^{\fr14}\mu^{-1}+\nu^{-\frac18}\mu^{-\frac13})\left(\sup_{0\leq\tau\leq t}\mathfrak{E}_{\theta}(\tau)\right) \left(\sup_{0\leq\tau\leq t}\mathfrak{E}_{\omega}(\tau) \right)^{\frac12}\n\\ 
	&+\mu^{-\fr12}\left(\sup_{0\leq\tau\leq t}\mathfrak{E}_{\theta}(\tau) \right)^{\frac32}+\mu^{-\frac{2}{3}}\|u_{0,in}^1\|^{\fr12}\left(\sup_{0\leq\tau\leq t}\mathfrak{E}_{\omega}(\tau) \right)^{\frac{1}{4}}\left(\sup_{0\leq\tau\leq t}\mathfrak{E}_{\theta}(\tau) \right)+\mathfrak{E}_{\theta}(0) .
\end{align}
With \eqref{2.21}--\eqref{2.22}, it is straightforward to verify from \eqref{1.4}--\eqref{1.5} that
\begin{equation}\notag
	\mathfrak{E}_{\omega}(t) \lesssim \varepsilon_0\min\{\nu^{\fr34},\mu^{\fr34},\nu^{-\fr14}\mu\},\quad~~\mathfrak{E}_{\theta}(t) \lesssim \varepsilon_0\min\{\nu^{\fr94},\mu^{\fr94}\},
\end{equation}
which implies that \eqref{1.6} and \eqref{1.7} hold. The proof of Theorem \ref{thm1.1} is completed. 

The proof of Theorem \ref{thm1.2} is similar to Theorem \ref{thm1.1}, except that we can use Poincar\'e inequality at this time, so we can also get the following estimates.
\begin{prop}\label{p3}
Under the conditions of Theorem \ref{thm1.2}, we have
	\begin{equation}\label{2.23}
	\ddt\|u_0^1\|^2+\nu\|\de_yu_0^1\|^2\lesssim	\mathcal{E}_{\omega,0}^{\fr12}\mathcal{D}_{\omega,\neq},
	\end{equation}
	\begin{equation}\label{2.24}
\ddt\left({\rm e}^{2\delta_0\kappa t}\mathcal{E}_{\omega,0}\right)+\fr13{\rm e}^{2\delta_0\kappa t}\mathcal{D}_{\omega,0}\lesssim \nu^{-\fr12}\left({\rm e}^{2\delta_0\kappa t}\mathcal{E}_{\omega,\neq}\right)^{\fr12}\left({\rm e}^{2\delta_0\kappa t}\mathcal{D}_{\omega,0}\right)^{\fr12}\mathcal{D}_{\omega,\neq}^{\fr12},		
	\end{equation}
\begin{align}\label{2.25}
\ddt\left({\rm e}^{2\delta_0\kappa t}\mathcal{E}_{\theta,0}\right)+\fr13{\rm e}^{2\delta_0\kappa t}\mathcal{D}_{\theta,0}\lesssim& (\mu^{-\fr12}+\nu^{-\fr14}\mu^{-\fr14})\left({\rm e}^{2\delta_0\kappa t}\mathcal{E}_{\theta,\neq}\right)^{\fr12}\left({\rm e}^{2\delta_0\kappa t}\mathcal{D}_{\theta,0}\right)^{\fr12}\mathcal{D}_{\omega,\neq}^{\fr12}\n\\ 
&+\nu^{\fr14}\mu^{-\fr34}\mathcal{E}_{\om,\neq}^{\fr12}\left({\rm e}^{2\delta_0\kappa t}\mathcal{D}_{\theta,0}\right)^{\fr12}\mathcal{D}_{\theta,\neq}^{\fr12}.		
	\end{align}
\end{prop}

\begin{prop}\label{p4}
Under the conditions of Theorem \ref{thm1.2}, we have \eqref{2.16}--\eqref{2.17}.
\end{prop}
{\bf Proof of Theorem \ref{thm1.2}:} Firstly, for any $k \in \mathbb{Z}$, we define the following two energy functionals:
\begin{equation}\label{2.26}
	\bar{\mathfrak{E}_{\om}}:={\rm e}^{2\delta_0\kappa t}\mathcal{E}_{\omega,0}+\mathcal{E}_{\om,\neq}+\fr14\int_0^t({\rm e}^{2\delta_0\kappa \tau}\mathcal{D}_{\om,0}+\mathcal{D}_{\om,\neq})(\tau){\rm d}\tau,
\end{equation}
\begin{equation}\label{2.27}
	\bar{\mathfrak{E}_{\theta}}:={\rm e}^{2\delta_0\kappa t}\mathcal{E}_{\theta,0}+\mathcal{E}_{\theta,\neq}+\frac{1}{2^{12}}\int_0^t({\rm e}^{2\delta_0\kappa \tau}\mathcal{D}_{\theta,0}+\mathcal{D}_{\theta,\neq})(\tau){\rm d}\tau.
\end{equation} 
Then by combining Propositions \ref{p3}--\ref{p4}, we can conclude Theorem \ref{thm1.2}.  

\begin{rem}
   Because the Poiseuille flow is more complex than the couette flow, we have to deal with more nonlinear terms in the process of constructing energy estimation, and it is also difficult to utilize the multiplier theory. 
   Therefore, we need to put forward more requirements for the initial values. For example, from \eqref{2.16}--\eqref{2.17}, we can konw that the estimation of $u_0^1$ is essential. In addition, we will also need to use the estimates of $y\de_y\theta_0$ and $y\de_y\om_0$, see \eqref{i51} for instance.
    
\end{rem}

\section{Estimate of zero mode}\label{s3}

In this section, we are going to prove Propositions \ref{p1} and \ref{p3}.
Let's first analyze the zero mode of the system \eqref{1.3}:
\begin{eqnarray}\label{zero}
\left\{\begin{aligned}
&\partial_t u_0^1  -\nu \partial_{yy}u_0^1= -(u \cdot \nabla u^1)_0= (\nabla^{\bot} \phi \cdot \nabla \de_y\phi)_0 ,\\
&\partial_t \omega_0  -\nu \partial_{yy}\omega_0= - (u \cdot \nabla \omega)_0= - (\nabla^{\bot} \phi \cdot \nabla \omega)_0 ,\\
& \partial_t \theta_0 -\mu \partial_{yy}\theta_0=- (u \cdot \nabla \theta)_0=    - (\nabla^{\bot} \phi \cdot \nabla \theta)_0,\\
&\phi_0(y\in \de D)=\omega_0(y\in \de D)=0,\ \ \theta_0(y\in \de D)=0.
\end{aligned}\right.
\end{eqnarray}

{\bf Proof of \eqref{2.13a} and \eqref{2.23}:} Taking the inner product of \eqref{zero}$_1$ with $u_0^1$ in $L^2$, and then taking
 the real part, we have
\begin{equation}\label{3.2}
	\fr12\ddt\|u_0^1\|^2+\nu\|\de_y u_0^1\|^2=\RR\l u_0^1,(\nabla^{\bot} \phi \cdot \nabla \de_y\phi)_0\r,	
	\end{equation}
where we used that $\de_yu_0^1=0$ on the boundary $y\in\de D$.
Notice that $\nabla^{\bot} \phi=(-\de_y\phi, \de_x\phi)$ and $m>\fr12$, by using H$\ddot{{\rm o}}$lder inequality and Sobolev inequality, we can estimate
\begin{align}\n
\RR\l u_0^1,(\nabla^{\bot} \phi \cdot \nabla \de_y\phi)_0\r 
=&\RR\l u_0^1,-\sum_{l \in \Z} \partial_y \phi_l \cdot i(-l)  \de_y\phi_{-l}+ \sum_{l \in \Z} il \phi_l \cdot \partial_{yy} \phi_{-l}\r\n\\ 
=&\RR\l u_0^1,\sum_{l \in \Z} \partial_y \phi_l \cdot il  \de_y\phi_{-l}+ \de_y(\sum_{l \in \Z} il \phi_l \cdot \partial_{y} \phi_{-l})-\sum_{l \in \Z} il \de_y\phi_l \cdot \partial_{y} \phi_{-l}\r\n\\ 
=&\RR\l u_0^1,\de_y(\sum_{l \in \Z} il \phi_l \cdot \partial_{y} \phi_{-l})\r=-\RR\l \de_yu_0^1,\sum_{l \in \Z} il \phi_l \cdot \partial_{y} \phi_{-l}\r\n\\ 
\lesssim&\|\de_yu_0^1\|	\sum_{k \neq 0}|k| \|\phi_k\|_{L^\infty} \|\partial_{y} \phi_{-k}\|\lesssim\|\de_yu_0^1\|	\sum_{k \neq 0}|k| \|\phi_k\|^{\fr12}\|\de_y\phi_k\|^{\fr12} \|\partial_{y} \phi_{-k}\|\n\\
\lesssim&\|\om_0\|	\left(\sum_{k \neq 0}|k| \||k|\phi_k\|\|\de_y\phi_k\|\right)^{\fr12} \left(\sum_{k \neq 0}\|\partial_{y} \phi_{k}\|^2\right)^{\fr12}\n\\
\lesssim&\|\om_0\|	\left(\sum_{k \neq 0}|k|^{2m+1}\|\nabla_k\phi_k\|^2\right)^{\fr12} \left(\sum_{k \neq 0}|k|^{2m}\|\partial_{y} \phi_{k}\|^2\right)^{\fr12}\n\\
\lesssim&\mathcal{E}_{\omega,0}^{\fr12}\mathcal{D}_{\omega,\neq}.\n
\end{align}
So we obtain \eqref{2.13a} and \eqref{2.23}.

{\bf Proof of \eqref{2.15} and \eqref{2.25}:}  By using a similar procedure as \eqref{3.2}, one can compute that
\begin{align}\label{3.3}
\fr12\ddt\|\theta_0\|^2+\mu\|\de_y\theta_0\|^2=-\RR\l\theta_0,(\nabla^{\bot} \phi \cdot \nabla \theta)_0\r,	
	\end{align}
    
\begin{align}\label{3.4}
\fr12\ddt\|\de_y\theta_0\|^2+\mu\|\de_{yy}\theta_0\|^2=-\RR\l\de_y\theta_0,\de_y(\nabla^{\bot} \phi \cdot \nabla \theta)_0\r,	
	\end{align}
    
\begin{align}\label{3.5}
	\fr12\ddt\|y\de_y\theta_0\|^2=\mu\RR\l y^2\de_y\theta_0,\de_y^3\theta_0\r-\RR\l y^2\de_y\theta_0,\de_y(\nabla^{\bot} \phi \cdot \nabla \theta)_0\r.	
	\end{align}
Here we used  integration by parts and $\theta_0=\de_{yy}\theta_0=0$ on the boundary $y\in\de D$. By H$\ddot{{\rm o}}$lder inequality and Sobolev inequality,
we have
\begin{align}\label{3.6}
-\RR\l\theta_0,(\nabla^{\bot} \phi \cdot \nabla \theta)_0\r=&\RR\langle \theta_0, \sum_{l \in \Z} \partial_y \phi_l \cdot i(-l)  \theta_{-l} - \sum_{l \in \Z} il \phi_l \cdot \partial_{y} \theta_{-l}  \rangle\n\\ 
=&\RR\langle \theta_0, \sum_{l \in \Z} \partial_y \phi_l \cdot i(-l)  \theta_{-l} - \de_y\left(\sum_{l \in \Z} il \phi_l \cdot \theta_{-l}\right)+\sum_{l \in \Z} il \de_y\phi_l \cdot \theta_{-l}  \rangle\n\\
=&\RR\langle \theta_0, - \de_y\left(\sum_{l \in \Z} il \phi_l \cdot \theta_{-l}\right)\rangle=\RR\langle \de_y\theta_0, \sum_{l \in \Z} il \phi_l \cdot \theta_{-l}\rangle\n\\
\lesssim&\|\de_y\theta_0\|	\sum_{k \neq 0}|k| \|\phi_k\|_{L^\infty} \|\theta_{-k}\|\lesssim\|\de_y\theta_0\|	\sum_{k \neq 0}|k| \|\phi_k\|^{\fr12}\|\de_y\phi_k\|^{\fr12} \|\theta_{-k}\|\n\\
\lesssim&\|\de_y\theta_0\|	\left(\sum_{k \neq 0}|k| \||k|\phi_k\|\|\de_y\phi_k\|\right)^{\fr12} \left(\sum_{k \neq 0}\|\theta_{k}\|^2\right)^{\fr12}\n\\
\lesssim&\mu^{-\fr12}\left(\mu^{\fr12}\|\de_y\theta_0\|^2\right)	\left(\sum_{k \neq 0}|k|^{2m+1}\|\nabla_k\phi_k\|^2\right)^{\fr12} \left(\sum_{k \neq 0}|k|^{2m}\|\theta_{k}\|^2\right)^{\fr12}\n\\
\lesssim&\mu^{-\fr12}\mathcal{D}_{\theta,0}^{\fr12}\mathcal{D}_{\omega,\neq}^{\fr12}\mathcal{E}_{\theta,\neq}^{\fr12},	
\end{align}
\begin{align}\label{3.7}
&-\RR\l\de_y\theta_0,\de_y(\nabla^{\bot} \phi \cdot \nabla \theta)_0\r=\RR\l\de_{yy}\theta_0,(\nabla^{\bot} \phi \cdot \nabla \theta)_0\r\n\\ 
=&\RR\langle \de_{yy}\theta_0, \de_y\left(\sum_{l \in \Z} il \phi_l \cdot \theta_{-l}\right)\rangle=\RR\langle \de_{yy}\theta_0, \sum_{l \in \Z} il \de_y\phi_l \cdot \theta_{-l}+\sum_{l \in \Z} il \phi_l \cdot \de_y\theta_{-l}\rangle\n\\
\lesssim&\|\de_{yy}\theta_0\|	\left(\sum_{k \neq 0}|k| \|\de_y\phi_k\|_{L^\infty} \|\theta_{-k}\|+\sum_{k \neq 0}|k| \|\phi_k\|_{L^\infty} \|\de_y\theta_{-k}\|\right)\n\\ 
\lesssim&\|\de_{yy}\theta_0\|	\left(\sum_{k \neq 0}|k| \||k|\phi_k\|\|\de_y\phi_k\|\right)^{\fr12} \left(\sum_{k \neq 0}\|\de_y\theta_{k}\|^2\right)^{\fr12}\n\\ 
&+\|\de_{yy}\theta_0\|	\left(\sum_{k \neq 0}|k| \||k|\de_y\phi_k\|\|\de_{yy}\phi_k\|\right)^{\fr12} \left(\sum_{k \neq 0}\|\theta_{k}\|^2\right)^{\fr12}\n\\
\lesssim&\mu^{-1}\left(\mu^{\fr32}\|\de_{yy}\theta_0\|^2\right)^{\fr12}	\left(\sum_{k \neq 0}|k|\|\nabla_k\phi_k\|^2\right)^{\fr12} \left(\mu^{\fr12}\sum_{k \neq 0}|k|^{2m-\fr12}\|\de_y\theta_{k}\|^2\right)^{\fr12}\n\\ 
&+\nu^{-\fr14}\mu^{-\fr34}\left(\mu^{\fr32}\|\de_{yy}\theta_0\|^2\right)^{\fr12}\left(\nu^{\fr12}\sum_{k \neq 0}|k|^{2m+\fr12}\|\omega_k\|^2\right)^{\fr12} \left(\sum_{k \neq 0}\|\theta_{k}\|^2\right)^{\fr12}\n\\
\lesssim&(\mu^{-1}+\nu^{-\fr14}\mu^{-\fr34})\mathcal{D}_{\theta,0}^{\fr12}\mathcal{D}_{\omega,\neq}^{\fr12}\mathcal{E}_{\theta,\neq}^{\fr12},	
\end{align}
\begin{align}\label{3.8}
\mu\RR\l y^2\de_y\theta_0,\de_y^3\theta_0\r=-\mu\|y\de_{yy}\theta_0\|^2-2\mu\RR\l y\de_y\theta_0,\de_{yy}\theta_0\r=-\mu\|y\de_{yy}\theta_0\|^2+\mu\|\de_{y}\theta_0\|^2,	\end{align}
and
\begin{align}\label{3.9}
&-\RR\l y^2\de_y\theta_0,\de_y(\nabla^{\bot} \phi \cdot \nabla \theta)_0\r\n\\
=&\RR\langle y\de_{yy}\theta_0+2\de_{y}\theta_0, y\sum_{l \in \Z} il \de_y\phi_l \cdot \theta_{-l}+y\sum_{l \in \Z} il \phi_l \cdot \de_y\theta_{-l}\rangle\n\\
\lesssim&(\|y\de_{yy}\theta_0\|+\|\de_{y}\theta_0\|)\left(\sum_{k \neq 0}|k| \|\de_y\phi_k\|_{L^\infty} \|y\theta_{-k}\|+\sum_{k \neq 0}|k| \|\phi_k\|_{L^\infty} \|y\de_y\theta_{-k}\|\right)\n\\
\lesssim&\nu^{\fr14}\mu^{-\fr34}\mathcal{D}_{\theta,0}^{\fr12}\left(\nu^{-\fr12}\sum_{k \neq 0}|k|^{2m+\fr12}\|\nabla_k\phi_k\|^2\right)^{\fr12} \left(\mu^{\fr12}\sum_{k \neq 0}|k|^{2m+\fr12}\|y\de_y\theta_{k}\|^2\right)^{\fr12}\n\\ 
&+\nu^{-\fr14}\mu^{-\fr14}\mathcal{D}_{\theta,0}^{\fr12}\left(\nu^{\fr12}\sum_{k \neq 0}|k|^{2m+\fr12}\|\omega_k\|^2\right)^{\fr12} \left(\mu^{-\fr12}\sum_{k \neq 0}|k|^{2m+\fr12}\|y\theta_{k}\|^2\right)^{\fr12}\n\\
\lesssim&\nu^{\fr14}\mu^{-\fr34}\mathcal{D}_{\theta,0}^{\fr12}\mathcal{E}_{\omega,\neq}^{\fr12}\mathcal{D}_{\theta,\neq}^{\fr12}+\nu^{-\fr14}\mu^{-\fr14}\mathcal{D}_{\theta,0}^{\fr12}\mathcal{D}_{\omega,\neq}^{\fr12}\mathcal{E}_{\theta,\neq}^{\fr12}.	
\end{align}
Substitute \eqref{3.6}, \eqref{3.7} and \eqref{3.8}--\eqref{3.9} into \eqref{3.3}, \eqref{3.4} and \eqref{3.5} respectively, and calculate \eqref{3.3}$+\frac{1}{2}\times$\eqref{3.4}$+\mu^{\frac{1}{2}}\times$\eqref{3.5} to obtain \eqref{2.15}.

For \eqref{2.25}, notice that if $\delta_0\leq \frac{1}{24C_{\star}}$ ($C_{\star}$ is the Poincar\'e constant), we have from Poincar\'e inequality that 
\begin{equation}\label{3.10}
    2\delta_0\kappa\mathcal{E}_{\theta,0}\leq \frac{1}{6}\mathcal{D}_{\theta,0}.
\end{equation}
Then multiplying both sides of equations  \eqref{3.3}, \eqref{3.4} and \eqref{3.5} by ${\rm e}^{2\delta_0\kappa t}$, respectively. Through calculations similar to those in \eqref{3.6}--\eqref{3.9} and using the definitions \eqref{2.10}--\eqref{2.13} and \eqref{3.10}, we can get \eqref{2.25}.

{\bf Proof of \eqref{2.14} and \eqref{2.24}:} By direct calculation we have
\begin{equation}\n
	\fr12\ddt\|\om_0\|^2+\nu\|\de_y\om_0\|^2=-\RR\l\om_0,(\nabla^{\bot} \phi \cdot \nabla \omega)_0\r,	
	\end{equation}
\begin{equation}\n
	\fr12\ddt\|\de_y\om_0\|^2+\nu\|\de_{yy}\om_0\|^2=-\RR\l\de_y\om_0,\de_y(\nabla^{\bot} \phi \cdot \nabla \omega)_0\r,	
	\end{equation}
\begin{equation}\n
	\fr12\ddt\|y\de_y\om_0\|^2=\nu\RR\l y^2\de_y\om_0,\de_y^3\om_0\r-\RR\l y^2\de_y\om_0,\de_y(\nabla^{\bot} \phi \cdot \nabla \omega)_0\r.	
	\end{equation}
Similar to \eqref{3.6}--\eqref{3.9}, we can prove that
\begin{equation}
-\RR\l\om_0,(\nabla^{\bot} \phi \cdot \nabla \omega)_0\r
\lesssim\nu^{-\fr12}\mathcal{D}_{\omega,0}^{\fr12}\mathcal{D}_{\omega,\neq}^{\fr12}\mathcal{E}_{\omega,\neq}^{\fr12},\n	
\end{equation}
\begin{equation}
-\RR\l\de_y\om_0,\de_y(\nabla^{\bot} \phi \cdot \nabla \omega)_0\r
\lesssim\nu^{-1}\mathcal{D}_{\omega,0}^{\fr12}\mathcal{D}_{\omega,\neq}^{\fr12}\mathcal{E}_{\omega,\neq}^{\fr12},\n	
\end{equation}
\begin{equation}
\nu\RR\l y^2\de_y\om_0,\de_y^3\om_0\r=-\nu\|y\de_{yy}\om_0\|^2+\nu\|\de_{y}\om_0\|^2,\n	\end{equation}
\begin{equation}
-\RR\l y^2\de_y\om_0,\de_y(\nabla^{\bot} \phi \cdot \nabla \omega)_0\r
\lesssim\nu^{-\fr12}\mathcal{D}_{\omega,0}^{\fr12}\mathcal{D}_{\omega,\neq}^{\fr12}\mathcal{E}_{\omega,\neq}^{\fr12}.\n	
\end{equation}
 By combining the above formulas, and using H$\ddot{{\rm o}}$lder inequality, Sobolev inequality and Poincar\'e inequality, one can finally get \eqref{2.14} and \eqref{2.24}.   

Thus, the proof of Propositions \ref{p1} and \ref{p3} is now completed.

\section{Estimate of nonzero mode}\label{s4}

In this section, we study the nonlinear estimates for the nonzero mode of system \eqref{1.3}, and prove Propositions \ref{p2} and \ref{p4}.
Taking the Fourier transform in $x$, we get that for $k \neq0$,
\begin{align}\label{4.1}
\begin{cases}
\de_t \omega_k+iky^2\omega_k -2ik\phi_k-\nu\Delta_k \omega_k+(u\cdot \nabla \omega)_k=ik\theta_k,\\
\partial_t \theta_k+iky^2\theta_k+(u \cdot \nabla \theta)_k-\mu \Delta_k \theta_k=0, \\
u_k=(-\de_y,ik) \phi_k, \quad \Delta_k\phi_k=\omega_k,\\ 
\phi_k(y\in \de D)=\omega_k(y\in \de D)=0,\ \ \theta_k(y\in \de D)=0.
\end{cases}
\end{align}

\subsection{Estimate of $\theta$ } 
\begin{lem}\label{l4.1}
If $(\omega,\theta)(t,x,y)$ solves \eqref{4.1}, we have the following balances:
  \begin{equation}\label{4.2}
  \frac12\ddt \|\theta_k\|^2 +\mu\|\nabla_k\theta_k\|^2=-\RR\l\theta_k,(u\cdot \nabla \theta)_k\r,
  \end{equation}
  \begin{equation}\label{4.3}
  \frac12\ddt \|\de_y\theta_k\|^2 +\mu\|\de_y\nabla_k\theta_k\|^2+2\RR\l\de_y\theta_k, iky\theta_k\r=-\RR\l\de_y\theta_k,\de_y(u\cdot \nabla \theta)_k\r,
  \end{equation}
  \begin{align}\label{4.4}
  &\ddt\RR\l iky \theta_k,\de_y\theta_k\r+2|k|^2\| y\theta_k\|^2 \n\\
  =&-2\mu\RR\l iky\de_{y}\theta_k, \Delta_k \theta_k\r+2\RR\l iky\de_{y}\theta_k, (u\cdot \nabla \theta)_k\r+\RR\l ik\theta_k, (u\cdot \nabla \theta)_k\r,
  \end{align}
  \begin{equation}\label{4.5}
  \frac12\ddt\| y\theta_k\|^2+\mu\|y\nabla_k\theta_k\|^2=\mu\|\theta_k\|^2-\RR\l y^2\theta_k, (u\cdot \nabla \theta)_k\r.
  \end{equation}
\end{lem}
\begin{proof}
Taking the inner product of system \eqref{4.1}$_2$ with $\theta_k$ and $-\de_{yy}\theta_k$ in $L^2$ respectively, and integrating the resulting equation by parts, we can obtain \eqref{4.2}--\eqref{4.3}. 
For \eqref{4.4}, we use \eqref{4.1}$_2$ to compute
\begin{align}
\ddt\RR\l iky\theta_k,\de_y\theta_k\r=&\RR\l iky\de_{t} \theta_k,\de_y\theta_k\r+\RR\l iky\theta_k,\de_{yt}\theta_k\r=-2\RR\l iky\de_{y}\theta_k,\de_{t} \theta_k\r-\RR\l ik\theta_k,\de_{t}\theta_k\r\n\\ 
=&\RR\l 2iky\de_{y}\theta_k+ik\theta_k, iky^2\theta_k -\mu \Delta_k\theta_k+(u\cdot \nabla \theta)_k\r   \n\\
=&-\mu\RR\l 2iky\de_{y}\theta_k+ik\theta_k, \Delta_k \theta_k\r+\RR\l 2iky\de_{y}\theta_k+ik\theta_k,(u\cdot \nabla \theta)_k\r   \n\\
&+\RR\l 2iky\de_{y}\theta_k+ik\theta_k, iky^2\theta_k\r .
\end{align}
Notice that
\begin{align}
\RR\l 2iky\de_{y}\theta_k+ik\theta_k, iky^2\theta_k\r=2\RR\l iky^3\de_{y}\theta_k, ik\theta_k\r+\RR\l ik\theta_k, iky^2\theta_k\r=-2|k|^2\| y\theta_k\|^2,\notag	
\end{align}
then we have \eqref{4.4}. For \eqref{4.5}, by direct calculation, we have
\begin{align}
\frac12\ddt\| y\theta_k\|^2
&=\RR\l y\theta_k, y(\mu\Delta_k \theta_k-iky^2\theta_k -(u\cdot \nabla \theta)_k)\r\notag\\
&=\mu\l y\theta_k, y \Delta_k \theta_k\r -\RR\l y\theta_k, y(u\cdot \nabla \theta)_k)\r \notag\\
&=\mu\|\theta_k\|^2-\mu\|y\nabla_k\theta_k\|^2-\RR\l y\theta_k, y(u\cdot \nabla \theta)_k)\r,\n
 \end{align}
hence we have \eqref{4.5}.
The proof of Lemma \ref{l4.1} is over.
\end{proof}

With the Lemma \ref{l4.1} in hand, we are ready to prove \eqref{2.17} in Propositions \ref{p2} and \ref{p4}.

{\bf Proof of \eqref{2.17}:} 
Recalling  the definition of  $E_{\theta}[\theta_k]$ in \eqref{2.7}, and combining with Lemma \ref{l4.1}, we can deduce
\begin{align}\label{4.6}
&\ddt E_{\theta}[\theta_k] +\mu\|\nabla_k\theta_k\|^2+2^{-7}\mu^{\fr32}|k|^{-\fr12}\|\de_y\nabla_k\theta_k\|^2+2^{-10}|k|\| y\theta_k\|^2+2^{-9}\mu^{\fr12}|k|^{\fr12}\|y\nabla_k\theta_k\|^2\n\\
=&2^{-9}\mu^{\fr12}|k|^{\fr12}\|\theta_k\|^2-2^{-6}\mu^{\frac12}|k|^{-\frac12}\RR\l iky\theta_k,\de_y\theta_k\r-2^{-10}\mu|k|^{-1}\RR\l \Delta_k \theta_k, iky\de_{y}\theta_k\r\n\\ 
&-\RR\l\theta_k,(u\cdot \nabla \theta)_k\r-2^{-7}\mu^{\frac12}|k|^{-\frac12}\RR\l\de_y\theta_k,\de_y(u\cdot \nabla \theta)_k\r +2^{-10}|k|^{-1}\RR\l iky\de_{y}\theta_k,(u\cdot \nabla \theta)_k\r\n\\
&+2^{-11}|k|^{-1}\RR\l ik\theta_k, (u\cdot \nabla \theta)_k\r-2^{-9}\mu^{-\fr12}|k|^{\fr12}\RR\l y^2\theta_k,(u\cdot \nabla \theta)_k\r.
\end{align}
We first analyze the linear term on the right side of \eqref{4.6}.
In order to obtain the $L^2$ norm of $\theta$ on the left hand side of \eqref{4.6}, an integration by parts yields
\begin{equation}\n
 \l y \theta_k, \de_y\theta_k\r= -\frac12  \|\theta_k\|^2.
\end{equation}
It means that for every $\sigma>0$,
\begin{equation}\n
\norm{\theta_k}^2\leq 2 \norm{\de_y\theta_k}\norm{y \theta}\leq \sigma \norm{\de_y\theta_k}^2 +\frac{1}{\sigma}\norm{y \theta_k}^2.
\end{equation}
By taking $\sigma=2^{-\frac{13}{2}}\mu^{\frac12}|k|^{-\frac12}$, we conclude that
\begin{equation}\label{4.7}
 2^{-\frac{13}{2}}\mu^{\frac12}|k|^{\frac12}\norm{\theta_k}^2\leq \frac{\mu}{4}\norm{\de_y\theta_k}^2+2^{-11}|k|\norm{y\theta_k}^2.
\end{equation}
Then, through direct calculation, the second and third terms on the right side of \eqref{4.6} can be estimated as
\begin{align}\label{4.8}
-2^{-6}\mu^{\frac12}|k|^{-\frac12}\RR\l iky\theta_k,\de_y\theta_k\r\leq \frac{\mu}{4} \|\nabla_k \theta_k\|^2+ 2^{-12} |k|\|y\theta_k\|^2,
\end{align}
\begin{align}\label{4.9}
-2^{-10}\mu|k|^{-1}\RR\l \Delta_k \theta_k, iky\de_{y}\theta_k\r=&2\bar\beta\mu\RR\l y|k|^2 \theta_k, ik\de_{y}\theta_k\r-2\bar\beta\mu\RR\l \de_{yy} \theta_k, iky\de_{y}\theta_k\r\n\\ 
\leq& -2^{-9}\mu\|\de_{y} \nabla_k\theta_k\|\|y\nabla_k\theta_k\|  \n\\
\leq&2^{-9}\mu^{\frac32}|k|^{-\frac12}\|\de_{y} \nabla_k\theta_k\|^2  +2^{-11}\mu^{\fr12}|k|^{\fr12}\|y\nabla_k\theta_k\|^2.
\end{align}
Substituting \eqref{4.7}--\eqref{4.9} into \eqref{4.6}, and multiplying both sides of the resulting equation by ${\rm e}^{2\delta_0 \kappa^{\fr12}t}$, after performing some integration by parts, we obtain
\begin{align}\label{4.10}
&\ddt \mathcal{E}_{\theta,\neq} +2^{-12}\mathcal{D}_{\theta,\neq}-2\delta_0\kappa^{\fr12 }\mathcal{E}_{\theta,\neq}\n\\
\leq&-{\rm e}^{2\delta_0 \kappa^{\fr12}t}\sum_{k \neq 0}|k|^{2m}\RR\l\theta_k,(u\cdot \nabla \theta)_k\r-2^{-7}\mu^{\frac12}{\rm e}^{2\delta_0 \kappa^{\fr12}t}\sum_{k \neq 0}|k|^{2m-\fr12}\RR\l\de_y\theta_k,\de_y(u\cdot \nabla \theta)_k\r \n\\ 
&+2^{-10}{\rm e}^{2\delta_0 \kappa^{\fr12}t}\sum_{k \neq 0}|k|^{2m-1}\RR\l iky\de_{y}\theta_k,(u\cdot \nabla \theta)_k\r+2^{-11}{\rm e}^{2\delta_0 \kappa^{\fr12}t}\sum_{k \neq 0}|k|^{2m-1}\RR\l ik\theta_k, (u\cdot \nabla \theta)_k\r\n\\
&-2^{-9}\mu^{-\fr12}{\rm e}^{2\delta_0 \kappa^{\fr12}t}\sum_{k \neq 0}|k|^{2m+\fr12}\RR\l y^2\theta_k,(u\cdot \nabla \theta)_k\r:=\sum_{i=1}^5I_i.
\end{align}
Firstly, notice that if $\delta_0\leq 2^{-16}$, 
\begin{equation}\label{4.45}
 2\delta_0\kappa^{\fr12 }\mathcal{E}_{\theta,\neq}\leq 2^{-14}\mathcal{D}_{\theta,\neq}.   
\end{equation}
Next, we estimate the terms on the right-hand side of \eqref{4.10} one by one.

{\it \underline{Estimate of $I_1$:}}
We can rewrite it as follows
\begin{align}\label{I1}
I_1=&{\rm e}^{2\delta_0 \kappa^{\fr12}t}\sum_{k \neq 0}|k|^{2m}\RR\langle \theta_k, \sum_{l \in \Z} \partial_y \phi_{k-l} \cdot il  \theta_{l}- \sum_{l \in \Z} i(k-l)\phi_{k-l} \cdot \partial_{y} \theta_{l}  \rangle\n\\
=&{\rm e}^{2\delta_0 \kappa^{\fr12}t}\sum_{k \neq 0}|k|^{2m} \RR\langle \theta_k, \partial_{y} \phi_{0} \cdot ik \theta_k \rangle -{\rm e}^{2\delta_0 \kappa^{\fr12}t}\sum_{k \neq 0}|k|^{2m} \RR\langle \theta_k,  ik \phi_{k} \cdot  \partial_{y} \theta_0 \rangle\n\\
&-{\rm e}^{2\delta_0 \kappa^{\fr12}t}\sum_{\substack{ k,l \neq 0 \\ k \neq l }} |k|^{2m} \RR\langle \theta_k,  \nabla^\perp_{k-l} \phi_{k-l} \cdot \nabla_l \theta_l \rangle\n\\
 :=& I_1^1+I_1^2+I_1^3.	
\end{align}
For the term $ I_1^1$, one can directly obtain
		\begin{align}\label{i11}
			 I_1^1&={\rm e}^{2\delta_0 \kappa^{\fr12}t}\sum_{k \neq 0}|k|^{2m} \RR \int_{\mathbb{R}}  \theta_k\cdot \overline{ ik\theta_k \partial_{y} \phi_{0}} \mathrm{d} y\n\\
			  &= -{\rm e}^{2\delta_0 \kappa^{\fr12}t}\sum_{k \neq 0}|k|^{2m} \RR \int_{\mathbb{R}}   ik |\theta_k|^2 \cdot \partial_{y} \phi_{0} \mathrm{d} y  =0.
		\end{align}
		For $I_1^2$, applying H\"older inequality and definitions \eqref{2.2}--\eqref{2.13} yields
		\begin{align}\label{i12}
			I_1^2&\lesssim{\rm e}^{2\delta_0 \kappa^{\fr12}t}\sum_{k \neq 0}(|k|^{m} \|\theta_k\|) (|k|^{m+1}\|\phi_k\|_{L^\infty}) \| \partial_{y} \theta_0\|\n\\
			&\lesssim \mu^{-\frac{1}{2}}{\rm e}^{2\delta_0 \kappa^{\fr12}t}\left(\sum_{k \neq 0} |k|^{2m} \|\theta_k\|^2\right)^{\fr12}  \left(\sum_{k \neq 0}|k|^{2m+2} \|\phi_k\|_{L^\infty}^2\right)^{\fr12}\left(\mu\| \partial_{y} \theta_0\|^2\right)^{\fr12}\n\\
			& \lesssim \mu^{-\frac{1}{2}} \mathcal{E}^{\frac{1}{2}}_{\theta,\neq} \mathcal{D}^{\frac{1}{2}}_{\omega, \neq}\mathcal{D}^{\frac{1}{2}}_{\theta, 0} ,
		\end{align}
		where we have used 
		\begin{align}\label{psi}
		{\rm e}^{2\delta_0 \kappa^{\fr12}t}\sum_{k\neq 0} |k|^{2m+2} \|\phi_{k} \|^2_{L^\infty} \lesssim& {\rm e}^{2\delta_0 \kappa^{\fr12}t}\sum_{k\neq 0}   (|k|^{m+\fr32} \|\phi_k\|) (|k|^{m+\fr12} \|\partial_{y}\phi_k\|) \n\\ 
		\lesssim&{\rm e}^{2\delta_0 \kappa^{\fr12}t}\left(\sum_{k\neq 0}|k|^{2m+3} \|\phi_k\|^2\right)^{\fr12} \left(\sum_{k\neq 0}(|k|^{2m+1} \|\partial_{y}\phi_k\|^2\right)^{\fr12}\n\\ 
		\lesssim& \mathcal{D}_{\omega, \neq}.
		\end{align}
For $ I_1^3$, we need to decompose the frequency $k$. First, denote 
\begin{align}\label{A1}
	&A_1(k,l):=\left\{  (k,l) \in \Z\setminus\{0\}\ \text{and} \ k \neq l~\big| |k-l|\leq \frac{|l|}{2}\   \right\},\n\\ 
&A_2(k,l):=\left\{(k,l) \in \Z\setminus\{0\}\ \text{and} \ k \neq l~\big| |k-l|> \frac{|l|}{2}\ \right\}.
\end{align}
We define the following two characteristic functions:
 		\begin{align*}
 			\chi_1(k,l)= \begin{cases}1, & \text { if } (k,l) \in A_1 \\ 0, & \text { if } (k,l) \in A_2\end{cases} \quad \text{and} \quad
 			\chi_2(k,l)= \begin{cases}0, & \text { if } (k,l) \in A_1 \\ 1, & \text { if } (k,l) \in A_2\end{cases}.
 		\end{align*}
 Notice that for $(k,l) \in A_1$, we have $|k|\leq |k-l|+|l|\leq \fr32 |l|$ and $|l|\leq |k-l|+|k|\leq \frac{|l|}{2}+|k|$, that is $|k|\approx|l|$; and for $(k,l) \in A_2$, we have $|l|< 2|k-l|$ and $|k|\leq |k-l|+|l|< 3|k-l|$. 
\begin{align}\label{i13}
	I_1^3=&{\rm e}^{2\delta_0 \kappa^{\fr12}t}\sum_{\substack{ k,l \neq 0 \\ k \neq l }} |k|^{2m} \RR\langle \theta_k, \partial_y \phi_{k-l} \cdot il  \theta_{l}\rangle-{\rm e}^{2\delta_0 \kappa^{\fr12}t}\sum_{\substack{ k,l \neq 0 \\ k \neq l }} |k|^{2m} \RR\langle \theta_k, i(k-l)\phi_{k-l} \cdot \partial_{y} \theta_{l}  \rangle\n\\
	=&-{\rm e}^{2\delta_0 \kappa^{\fr12}t}\sum_{\substack{ k,l \neq 0 \\ k \neq l }} |k|^{2m} \RR\langle \theta_k, \de_y(i(k-l)\phi_{k-l} \cdot \theta_{l})  \rangle+{\rm e}^{2\delta_0 \kappa^{\fr12}t}\sum_{\substack{ k,l \neq 0 \\ k \neq l }} |k|^{2m} \RR\langle \theta_k, ik\partial_y \phi_{k-l} \cdot  \theta_{l}\rangle\n\\ 
	:=&I_1^{3,1}+I_1^{3,2}.
\end{align}
We use the two characteristic functions $\chi_1$ and $\chi_2$ defined in \eqref{A1} to
 decompose $I_1^{3,1}$ into two parts,  and then apply \eqref{psi} and $m>\fr12$ to get
\begin{align}\label{i131}
	I_1^{3,1} &= {\rm e}^{2\delta_0 \kappa^{\fr12}t}\sum_{\substack{ k,l \neq 0 \\ k \neq l }} |k|^{2m} \RR\langle \partial_y\theta_k,   i(k-l)\phi_{k-l} \cdot  \theta_l \rangle  \n\\
	&\lesssim{\rm e}^{2\delta_0 \kappa^{\fr12}t}\sum_{\substack{ k,l \neq 0 \\ k \neq l }} |k|^{2m}\left(\chi_1(k,l)+\chi_2(k,l)\right)\|\partial_y\theta_k\|\|(k-l)\phi_{k-l}\|_{L^\infty} \|\theta_l \|\n\\ 
&\lesssim {\rm e}^{2\delta_0 \kappa^{\fr12}t}\sum_{\substack{k,l \neq 0 \\ k \neq l}} |k|^{m}\big( |l|^m+|k-l|^m \big) \|\partial_{y}\theta_k\|\| (k-l) \phi_{k-l}\|_{L^\infty} \| \theta_l\|\n\\
&\lesssim {\rm e}^{2\delta_0 \kappa^{\fr12}t}\Big( \sum_{k\neq 0} |k|^{2m} \|\partial_{y} \theta_k \|^2  \Big)^{\frac{1}{2}}\Big( \sum_{k\neq 0} |k|^{2m} \|\theta_{k} \|^2 \Big)^{\frac{1}{2}} \Big( \sum_{k\neq 0} |k| \|\phi_{k} \|_{L^\infty}  \Big)\n\\
			& \quad + {\rm e}^{2\delta_0 \kappa^{\fr12}t}\Big( \sum_{k\neq 0} |k|^{2m} \|\partial_{y} \theta_k \|^2  \Big)^{\frac{1}{2}}\Big( \sum_{k\neq 0} |k|^{2m+2} \|\phi_{k} \|^2_{L^\infty}  \Big)^{\frac{1}{2}} \Big( \sum_{k\neq 0} \|\theta_{k} \|  \Big)\n\\
&\lesssim {\rm e}^{2\delta_0 \kappa^{\fr12}t}\Big( \sum_{k\neq 0} |k|^{2m} \|\partial_{y} \theta_k \|^2  \Big)^{\frac{1}{2}}\Big( \sum_{k\neq 0} |k|^{2m} \|\theta_{k} \|^2 \Big)^{\frac{1}{2}} \Big( \sum_{k\neq 0} |k|^{2m+2} \|\phi_{k} \|^2_{L^\infty}  \Big)^{\frac{1}{2}}  \Big( \sum_{k\neq 0} |k|^{-2m}   \Big)^{\frac{1}{2}}\n\\
			& \quad + {\rm e}^{2\delta_0 \kappa^{\fr12}t}\Big( \sum_{k\neq 0} |k|^{2m} \|\partial_{y} \theta_k \|^2  \Big)^{\frac{1}{2}}\Big( \sum_{k\neq 0} |k|^{2m+2} \|\phi_{k} \|^2_{L^\infty}  \Big)^{\frac{1}{2}} \Big( \sum_{k\neq 0} |k|^{2m}\|\theta_{k} \|^2 \Big)^{\frac{1}{2}}\Big( \sum_{k\neq 0} |k|^{-2m}   \Big)^{\frac{1}{2}}\n\\
			& \lesssim  \mu^{-\frac{1}{2}}\mathcal{E}_{\theta,\neq}^{\frac{1}{2}} \mathcal{D}_{\theta, \neq}^{\frac{1}{2}} \mathcal{D}_{\omega, \neq}^{\frac{1}{2}}.
 		\end{align}
Similar to $I_1^{3,1}$, we also decompose $I_1^{3,2}$ into two parts, namely
 \begin{align}\label{i132}
	I_1^{3,2} &\lesssim {\rm e}^{2\delta_0 \kappa^{\fr12}t}\sum_{\substack{ k,l \neq 0 \\ k \neq l }}|k|^{2m+1} \| \theta_{k}\|_{L^2} \| \partial_{y} \phi_{k-l}\|_{L^\infty} \|\theta_{l} \|_{L^2}\n\\
			&\lesssim {\rm e}^{2\delta_0 \kappa^{\fr12}t}\sum_{\substack{ k,l \neq 0 \\ k \neq l }}|k|^{2m+1} \Big( \chi_1(k,l) + \chi_2(k,l) \Big)|k|^{2m+1} \| \theta_{k}\|_{L^2} \| \partial_{y} \phi_{k-l}\|_{L^\infty} \|\theta_{l} \|_{L^2}\n\\
			&:= I_1^{3,2,\chi_1} + I_1^{3,2,\chi_2}.
		\end{align}
		If $ (k,l)\in A_1$, then $ |k|\approx |l|$. By direct calculation, we get
		\begin{align}\label{i1321}
		I_1^{3,2,\chi_1} &\lesssim {\rm e}^{2\delta_0 \kappa^{\fr12}t}\sum_{\substack{ k,l \neq 0 \\ k \neq l }}(|k|^{m+\frac{1}{4}} \|\theta_{k} \|) (|l|^{m+\frac{3}{4}} \|\theta_{l} \|) \| \partial_{y} \phi_{k-l}\|_{L^\infty} \n\\
			&\lesssim \mu^{-\frac{1}{4}}{\rm e}^{\delta_0 \kappa^{\fr12}t}\mathcal{D}_{\theta,\neq}^{\frac{1}{2}}\Big( \sum_{k\neq 0} |k|^{2m+\frac{3}{2}} \| \theta_{k}\|^2\Big)^{\frac{1}{2}} \sum_{k\neq 0}\| \partial_{y} \phi_{k}\|_{L^\infty}\n\\
	&\lesssim \mu^{-\frac{1}{4}}{\rm e}^{\delta_0 \kappa^{\fr12}t}\mathcal{D}_{\theta,\neq}^{\frac{1}{2}}\Big( \sum_{k\neq 0} (|k|^{2m+\frac{1}{2}} \| \theta_{k}\|^2)^{\fr13}(|k|^{2m+2} \| \theta_{k}\|^2)^{\fr23}\Big)^{\frac{1}{2}} \Big( \sum_{k\neq 0} |k|^{2m}\| \partial_{y} \phi_{k}\|_{L^\infty}^2\Big)^{\frac{1}{2}}\Big( \sum_{k\neq 0} |k|^{-2m}\Big)^{\frac{1}{2}}\n\\
	&\lesssim \mu^{-\frac{1}{4}}{\rm e}^{\delta_0 \kappa^{\fr12}t}\mathcal{D}_{\theta,\neq}^{\frac{1}{2}}\Big( \sum_{k\neq 0}|k|^{2m+\frac{1}{2}} \| \theta_{k}\|^2\Big)^{\frac{1}{6}}\Big( \sum_{k\neq 0}|k|^{2m+2} \| \theta_{k}\|^2\Big)^{\frac{1}{3}} \Big( \sum_{k\neq 0} |k|^{2m}\|\om_{k}\|^2\Big)^{\frac{1}{2}}\n\\
			& \lesssim \mu^{-\frac{2}{3}} \mathcal{D}_{\theta,\neq} \mathcal{E}^{\frac{1}{2}}_{\omega,\neq}.
		\end{align}
			If $ (k,l)\in A_2$, then $ |k|\lesssim |k-l|+ |l|\lesssim |k-l|$. We can estimate
		\begin{align}\label{i1322}
		I_1^{3,2,\chi_2} &\lesssim {\rm e}^{2\delta_0 \kappa^{\fr12}t}\sum_{\substack{ k,l \neq 0 \\ k \neq l }}(|k|^{m+\frac{1}{4}} \|\theta_{k} \|) (|k-l|^{m+\frac{3}{4}} \| \partial_{y} \phi_{k-l}\|_{L^\infty}  \|\theta_{l} \| \n\\
			&\lesssim \mu^{-\frac{1}{4}}{\rm e}^{\delta_0 \kappa^{\fr12}t}\mathcal{D}_{\theta,\neq}^{\frac{1}{2}}\Big( \sum_{k\neq 0} |k|^{2m+\frac{1}{2}} \||k| \partial_{y} \phi_{k}\|\| \partial_{yy} \phi_{k}\| \Big)^{\frac{1}{2}} \Big( \sum_{k\neq 0} \| \theta_{k}\|\Big)\n\\
			&\lesssim \mu^{-\frac{1}{4}}{\rm e}^{\delta_0 \kappa^{\fr12}t}\mathcal{D}_{\theta,\neq}^{\frac{1}{2}}\Big( \sum_{k\neq 0} |k|^{2m+\frac{1}{2}} \| \om_{k}\|^2 \Big)^{\frac{1}{2}} \Big( \sum_{k\neq 0} |k|^{2m}\| \theta_{k}\|^2\Big)\n\\
			& \lesssim \nu^{-\frac{1}{4}} \mu^{-\frac{1}{4}}\mathcal{E}^{\frac{1}{2}}_{\theta,\neq} \mathcal{D}_{\theta,\neq}^{\frac{1}{2}} \mathcal{D}_{\omega,\neq}^{\frac{1}{2}}.
		\end{align}
Combining \eqref{I1}--\eqref{i12} and \eqref{i13}--\eqref{i1322}, we arrive at
\begin{align}\label{i1}
	I_1\lesssim\mu^{-\frac{1}{2}} \mathcal{E}^{\frac{1}{2}}_{\theta,\neq} \mathcal{D}^{\frac{1}{2}}_{\omega, \neq}\mathcal{D}^{\frac{1}{2}}_{\theta, 0}+(\mu^{-\frac{1}{2}}+\nu^{-\frac{1}{4}} \mu^{-\frac{1}{4}})\mathcal{E}_{\theta,\neq}^{\frac{1}{2}}  \mathcal{D}_{\omega, \neq}^{\frac{1}{2}}\mathcal{D}_{\theta, \neq}^{\frac{1}{2}}+\mu^{-\frac{2}{3}}\mathcal{E}^{\frac{1}{2}}_{\omega,\neq} \mathcal{D}_{\theta,\neq}.
\end{align}
{\it \underline{Estimate of $I_2$:}} Notice that
\begin{align}\label{I2}
	I_2&= 2^{-7}\mu^{\frac12}{\rm e}^{2\delta_0 \kappa^{\fr12}t}\sum_{k \neq 0}|k|^{2m-\fr12}\RR\langle \partial_{y} \theta_k, \partial_{y}\Big(\sum_{l \in \Z} \partial_{y} \phi_{k-l} \cdot il \theta_l\Big) \rangle \n\\
			&\quad - 2^{-7}\mu^{\frac12}{\rm e}^{2\delta_0 \kappa^{\fr12}t}\sum_{k \neq 0}|k|^{2m-\fr12}\RR\langle \partial_{y} \theta_k, \partial_{y} \Big(\sum_{l \in \Z} i(k-l) \phi_{k-l} \cdot \partial_{y} \theta_l \Big)\rangle\n\\
			& =2^{-7}\mu^{\frac12}{\rm e}^{2\delta_0 \kappa^{\fr12}t}\sum_{k \neq 0}|k|^{2m-\fr12}\big(\RR\langle \partial_{y} \theta_k, \partial_{y}\big(\partial_{y} \phi_{0} \cdot ik \theta_k\big) \rangle - \RR \langle \partial_{y} \theta_k,  \partial_{y} \big(ik \phi_{k} \cdot  \partial_{y} \theta_0\big)\rangle \big)   \n\\
			&\quad -2^{-7}\mu^{\frac12}{\rm e}^{2\delta_0 \kappa^{\fr12}t}\sum_{\substack{ k,l \neq 0 \\ k \neq l }}|k|^{2m-\fr12}\RR\langle \partial_{y} \theta_k,  \partial_{y} \big(\nabla^\bot_{k-l} \phi_{k-l} \cdot \nabla_l \theta_l\big) \rangle   \n\\
			& := I_2^1+I_2^2+I_2^3.	
\end{align}
For $k=l$, using H\"older inequality and definitions \eqref{2.2}--\eqref{2.13} gives
\begin{align}\label{i21}
I_2^1=&2^{-7}\mu^{\frac12}{\rm e}^{2\delta_0 \kappa^{\fr12}t}\sum_{k \neq 0}|k|^{2m-\fr12}\RR\langle \partial_{y} \theta_k, \partial_{yy} \phi_{0} \cdot ik \theta_k+\partial_{y} \phi_{0} \cdot ik \partial_{y}\theta_k\rangle\n\\ 
=&2^{-7}\mu^{\frac12}{\rm e}^{2\delta_0 \kappa^{\fr12}t}\sum_{k \neq 0}|k|^{2m-\fr12}\RR\langle \partial_{y} \theta_k, \partial_{yy} \phi_{0} \cdot ik \theta_k\rangle\n\\ 
\lesssim&\mu^{\frac12}{\rm e}^{2\delta_0 \kappa^{\fr12}t}\sum_{k \neq 0}(|k|^{m}\|\partial_{y} \theta_k\|)(|k|^{m+\fr12}\|\theta_k\|)\|\partial_{yy} \phi_{0}\|_{L^\infty}\n\\
\lesssim&{\rm e}^{\delta_0 \kappa^{\fr12}t}\mathcal{D}_{\theta,\neq}^{\fr12}\left(\sum_{k \neq 0}(|k|^{2m+\fr12}\|\theta_k\|^2)^{\fr23}(|k|^{2m+2}\|\theta_k\|^2)^{\fr13}\right)^{\fr12}\|\om_{0}\|^{\fr12}\|\de_y\om_{0}\|^{\fr12}\n\\
\lesssim&\nu^{-\frac18}\mu^{-\frac13}\mathcal{E}_{\om,0}^{\fr12}\mathcal{D}_{\theta,\neq}.	
\end{align}
For $l=0$, by \eqref{psi}, one has
\begin{align}\label{i22}
I_2^2=&2^{-7}\mu^{\frac12}{\rm e}^{2\delta_0 \kappa^{\fr12}t}\sum_{k \neq 0}|k|^{2m+\fr12}\RR \langle \partial_{yy} \theta_k, i\phi_{k} \cdot  \partial_{y} \theta_0\rangle   \n\\
\lesssim&\mu^{\frac12}{\rm e}^{2\delta_0 \kappa^{\fr12}t}\sum_{k \neq 0}(|k|^{m-\frac{1}{4}} \| \partial_{yy} \theta_k\|)(|k|^{m+\frac{3}{4}} \| \phi_{k}\|_{L^\infty}) \| \partial_{y} \theta_0\| \n\\
			\lesssim&\mu^{-\frac12}{\rm e}^{\delta_0 \kappa^{\fr12}t}\mathcal{E}_{\om,0}^{\fr12}\mathcal{D}_{\theta,\neq}^{\fr12}\left(\sum_{k \neq 0}|k|^{2m+\frac{3}{2}} \| \phi_{k}\|_{L^\infty}^2\right)^{\frac{1}{2}}\lesssim\mu^{-\frac12}\mathcal{E}_{\theta,0}^{\fr12}\mathcal{D}_{\theta,\neq}^{\fr12}\mathcal{D}_{\om,\neq}^{\fr12}.
\end{align}
Regarding $I_2^3$, it is noted that
\begin{align}\label{I23}
			I_2^3=& -2^{-7}\mu^{\frac12}{\rm e}^{2\delta_0 \kappa^{\fr12}t}\sum_{\substack{ k,l \neq 0 \\ k \neq l }}|k|^{2m-\fr12}\RR\langle \partial_{y} \theta_k,  \partial_{y} ( \partial_{y} \phi_{k-l} \cdot il \theta_l) -\partial_{y} (  i(k-l)\phi_{k-l} \cdot  \partial_{y}\theta_l) \rangle\n\\
			=&2^{-7}\mu^{\frac12}{\rm e}^{2\delta_0 \kappa^{\fr12}t}\sum_{\substack{ k,l \neq 0 \\ k \neq l }}|k|^{2m-\fr12}\RR\langle \partial_{yy} \theta_k,   \partial_{y} \phi_{k-l} \cdot il \theta_l \rangle \n\\
			&- 2^{-7}\mu^{\frac12}{\rm e}^{2\delta_0 \kappa^{\fr12}t}\sum_{\substack{ k,l \neq 0 \\ k \neq l }}|k|^{2m-\fr12}\left(\RR\langle \partial_{y} \theta_k,   i(k-l)\partial_{y}\phi_{k-l} \cdot  \partial_{y}\theta_l\rangle+ \RR \langle \partial_{y} \theta_k,   i(k-l)\phi_{k-l} \cdot  \partial_{yy}\theta_l\rangle\right)\n\\
			:=& I_2^{3,1}+I_2^{3,2}+I_2^{3,3}.
		\end{align}
By H\"older inequality, definitions \eqref{2.2}--\eqref{2.13} and \eqref{A1}, we can prove
\begin{align}\label{i231}
I_2^{3,1}\lesssim&\mu^{\frac12}{\rm e}^{2\delta_0 \kappa^{\fr12}t}\sum_{\substack{ k,l \neq 0 \\ k \neq l }}|k|^{2m-\fr12}\Big( \chi_1(k,l) + \chi_2(k,l) \Big)\| \partial_{yy} \theta_k\|\|  \theta_l\| \| \partial_{y} \phi_{k-l}\|_{L^\infty}\n\\
\lesssim&\mu^{\frac12}{\rm e}^{2\delta_0 \kappa^{\fr12}t}\sum_{\substack{ k,l \neq 0 \\ k \neq l }}(|k|^{m-\fr14}\| \partial_{yy} \theta_k\|)(|l|^{m+\fr34}\|  \theta_l\|) \| \partial_{y} \phi_{k-l}\|_{L^\infty}\n\\
&+\mu^{\frac12}{\rm e}^{2\delta_0 \kappa^{\fr12}t}\sum_{\substack{ k,l \neq 0 \\ k \neq l }}(|k|^{m-\fr14}\| \partial_{yy} \theta_k\|)(|l|^{\fr34}\|  \theta_l\|) (|k-l|^m\| \partial_{y} \phi_{k-l}\|_{L^\infty})\n\\
\lesssim&\mu^{-\frac14}{\rm e}^{\delta_0 \kappa^{\fr12}t}\mathcal{D}_{\theta,\neq}^{\fr12}\left(\sum_{k \neq 0}(|k|^{2m+\fr12}\|\theta_k\|^2)^{\fr13}(|k|^{2m+2}\|\theta_k\|^2)^{\fr23}\right)^{\fr12}	\left(\sum_{k \neq 0}|k|^{2m}\|\de_y\phi_k\|\|\de_{yy}\phi_k\|\right)^{\fr12}\n\\
\lesssim&\nu^{\frac18}\mu^{-\frac23}\mathcal{E}_{\om,\neq}^{\frac12}\mathcal{D}_{\theta,\neq},
\end{align}

 \begin{align}\label{i232}
I_2^{3,2}\lesssim&\mu^{\frac12}{\rm e}^{2\delta_0 \kappa^{\fr12}t}\sum_{\substack{ k,l \neq 0 \\ k \neq l }}|k|^{2m-\fr12}\Big( \chi_1(k,l) + \chi_2(k,l) \Big)\| \partial_{y} \theta_k\|\|\partial_{y} \theta_l\| |k-l|\| \partial_{y} \phi_{k-l}\|_{L^\infty}\n\\
\lesssim&\mu^{\frac12}{\rm e}^{2\delta_0 \kappa^{\fr12}t}\sum_{\substack{ k,l \neq 0 \\ k \neq l }}(|k|^{m}\| \partial_{y} \theta_k\|)(|l|^{m-\fr14}\|  \de_y\theta_l\|) (|k-l|^{\fr34}\| \partial_{y} \phi_{k-l}\|_{L^\infty})\n\\
&+\mu^{\frac12}{\rm e}^{2\delta_0 \kappa^{\fr12}t}\sum_{\substack{ k,l \neq 0 \\ k \neq l }}(|k|^{m}\| \partial_{y} \theta_k\|)\| \de_y \theta_l\| (|k-l|^{m+\fr12}\| \partial_{y} \phi_{k-l}\|_{L^\infty})\n\\
\lesssim&\mu^{-\frac14}\mathcal{E}_{\theta,\neq}^{\fr12}\mathcal{D}_{\theta,\neq}^{\fr12}\sum_{\substack{ k\neq 0}}|k|^{\fr34}\| \partial_{y} \phi_{k}\|_{L^\infty}+{\rm e}^{\delta_0 \kappa^{\fr12}t}\mathcal{D}_{\theta,\neq}^{\fr12}\left(\sum_{k \neq 0}|k|^{2m+1}\|\de_y\phi_k\|\|\de_{yy}\phi_k\|\right)^{\fr12}\sum_{\substack{ k\neq 0}}\| \partial_{y} \theta_k\|\n\\
\lesssim&\mu^{-\frac14}\mathcal{E}_{\theta,\neq}^{\fr12}\mathcal{D}_{\theta,\neq}^{\fr12}\left(\sum_{\substack{ k\neq 0}}(|k|^{m+1}\| \partial_{y} \phi_{k}\|)(|k|^{m+\fr12}\|\om_{k}\|)\right)^{\fr12}\n\\ 
&+{\rm e}^{\delta_0 \kappa^{\fr12}t}\mathcal{D}_{\theta,\neq}^{\fr12}\left(\sum_{k \neq 0}|k|^{2m}\|\om_k\|^2\right)^{\fr12}\left(\sum_{\substack{ k\neq 0}}|k|^{2m}\| \partial_{y} \theta_k\|^2\right)^{\fr12}\n\\
\lesssim&\nu^{-\frac14}\mu^{-\frac14}\mathcal{E}_{\theta,\neq}^{\frac12}\mathcal{D}_{\theta,\neq}^{\frac12}\mathcal{D}_{\om,\neq}^{\frac12}+\mu^{-\frac12}\mathcal{E}_{\theta,\neq}^{\frac12}\mathcal{D}_{\theta,\neq},
\end{align}
and
\begin{align}\label{i233}
I_2^{3,3}\lesssim&\mu^{\frac12}{\rm e}^{2\delta_0 \kappa^{\fr12}t}\sum_{\substack{ k,l \neq 0 \\ k \neq l }}|k|^{2m-\fr12}\Big( \chi_1(k,l) + \chi_2(k,l) \Big)\| \partial_{y} \theta_k\|\| \partial_{yy}\theta_l\| |k-l|\| \partial_{y} \phi_{k-l}\|_{L^\infty}\n\\
\lesssim&\mu^{\frac12}{\rm e}^{2\delta_0 \kappa^{\fr12}t}\sum_{\substack{ k,l \neq 0 \\ k \neq l }}(|k|^{m-\fr14}\| \partial_{y} \theta_k\|)(|l|^{m-\fr14}\| \partial_{yy}\theta_l\|) |k-l|\|\phi_{k-l}\|_{L^\infty}\n\\
&+\mu^{\frac12}{\rm e}^{2\delta_0 \kappa^{\fr12}t}\sum_{\substack{ k,l \neq 0 \\ k \neq l }}(|k|^{m-\fr14}\| \partial_{y} \theta_k\|)(|l|^{-\fr14}\|  \partial_{yy}\theta_l\|) (|k-l|^{m+1}\| \partial_{y} \phi_{k-l}\|_{L^\infty})\n\\
\lesssim&\mu^{\frac14}{\rm e}^{\delta_0 \kappa^{\fr12}t}\mathcal{E}_{\theta,\neq}^{\fr12}\left(\sum_{k \neq 0}(|k|^{2m-\fr12}\|\de_{yy}\theta_k\|^2\right)^{\fr12}\sum_{k \neq 0}|k|\|\phi_{k}\|_{L^\infty}\n\\ 
&+\mu^{\frac14}{\rm e}^{\delta_0 \kappa^{\fr12}t}\mathcal{E}_{\theta,\neq}^{\fr12}\left(\sum_{k \neq 0}|k|^{2m+2}\|\phi_{k}\|^2_{L^\infty}\right)^{\fr12}\sum_{k \neq 0}|k|^{-\fr14}\|\de_{yy}\theta_k\|\n\\
\lesssim&\mu^{\frac14}{\rm e}^{\delta_0 \kappa^{\fr12}t}\mathcal{E}_{\theta,\neq}^{\fr12}\left(\sum_{k \neq 0}(|k|^{2m-\fr12}\|\de_{yy}\theta_k\|^2\right)^{\fr12}\left(\sum_{k \neq 0}|k|^{2m+2}\|\phi_{k}\|^2_{L^\infty}\right)^{\fr12}\n\\
\lesssim&\mu^{-\frac12}\mathcal{E}_{\theta,\neq}^{\frac12}\mathcal{D}_{\theta,\neq}^{\frac12}\mathcal{D}_{\om,\neq}^{\frac12},
\end{align}
where we used that \eqref{psi} and $m>\fr12$.
By combining \eqref{I2}--\eqref{i233}, we have
\begin{align}\label{i2}
	I_2\lesssim& \nu^{-\frac18}\mu^{-\frac13}\mathcal{E}_{\om,0}^{\fr12}\mathcal{D}_{\theta,\neq}+\mu^{-\frac12}\mathcal{E}_{\theta,0}^{\fr12}\mathcal{D}_{\theta,\neq}^{\fr12}\mathcal{D}_{\om,\neq}^{\fr12}+\nu^{\frac18}\mu^{-\frac23}\mathcal{E}_{\om,\neq}^{\frac12}\mathcal{D}_{\theta,\neq}+\mu^{-\frac12}\mathcal{E}_{\theta,\neq}^{\frac12}\mathcal{D}_{\theta,\neq}\n\\ &+(\nu^{-\frac14}\mu^{-\frac14}+\mu^{-\frac12})\mathcal{E}_{\theta,\neq}^{\frac12}\mathcal{D}_{\theta,\neq}^{\frac12}\mathcal{D}_{\om,\neq}^{\frac12}.
\end{align}
{\it \underline{Estimate of $I_3$:}} We rewrite this term and divide it as follows
\begin{align}\label{I3}
I_3=&2^{-10}{\rm e}^{2\delta_0 \kappa^{\fr12}t}\sum_{k \neq 0}|k|^{2m-1}\RR\l iky\de_{y}\theta_k,-\sum_{l \in \Z} \partial_y \phi_{k-l} \cdot il  \theta_{l}+ \sum_{l \in \Z} i(k-l)\phi_{k-l} \cdot \partial_{y} \theta_{l}\r\n\\ 
=&2^{-10}{\rm e}^{2\delta_0 \kappa^{\fr12}t}\sum_{k \neq 0}|k|^{2m-1}\RR\l iky\de_{y}\theta_k, -\partial_{y} \phi_{0} \cdot ik \theta_k +ik \phi_{k} \cdot  \partial_{y} \theta_0\rangle\n\\ 
&+2^{-10}{\rm e}^{2\delta_0 \kappa^{\fr12}t}\sum_{k \neq 0}|k|^{2m-1}\RR\l iky\de_{y}\theta_k,  \nabla^\perp_{k-l} \phi_{k-l} \cdot \nabla_l \theta_l \rangle\n\\
 :=& I_3^1+I_3^2.		
\end{align}
By direct calculation we have 
\begin{align}\label{i31}
			I_3^1\lesssim&{\rm e}^{2\delta_0 \kappa^{\fr12}t}\sum_{k \neq 0}|k|^{m+\fr14}\|y\de_{y}\theta_k\|(|k|^{m+\fr34}\|\theta_k\|\|\partial_y \phi_{0}\|_{L^\infty}+|k|^{m+\fr34}\|\phi_k\|_{L^\infty}\| \partial_{y} \theta_0\|)\n\\
			\lesssim&\mu^{-\frac{1}{4}}{\rm e}^{\delta_0 \kappa^{\fr12}t}\mathcal{D}_{\theta,\neq}^{\fr12}\left(\sum_{k \neq 0}(|k|^{2m+\fr12}\|\theta_k\|^2)^{\fr13}(|k|^{2m+2}\|\theta_k\|^2)^{\fr23}\right)^{\fr12}\|\partial_y \phi_{0}\|_{L^\infty}\n\\ 
			&+\mu^{-\frac{1}{4}}{\rm e}^{\delta_0 \kappa^{\fr12}t}\mathcal{D}_{\theta,\neq}^{\fr12}\left(\sum_{k \neq 0}|k|^{2m+\fr12}\||k|\phi_k\|\|\de_y\phi_k\|\right)^{\fr12}\| \partial_{y} \theta_0\|\n\\
			\lesssim& \mu^{-\frac{2}{3}}\|u_{0}\|_{L^\infty}\mathcal{D}_{\theta,\neq}+\nu^{\frac{1}{4}}\mu^{-\frac{3}{4}} \mathcal{E}^{\frac{1}{2}}_{\om,\neq} \mathcal{D}^{\frac{1}{2}}_{\theta, \neq}\mathcal{D}^{\frac{1}{2}}_{\theta, 0}.
		\end{align}
	Applying \eqref{2.2}--\eqref{2.13}, \eqref{psi}, \eqref{A1} and $m>\fr12$, one has
	\begin{align}\label{i32}
		I_3^2\lesssim&{\rm e}^{2\delta_0 \kappa^{\fr12}t}\sum_{\substack{ k,l \neq 0 \\ k \neq l }}|k|^{2m}( \chi_1(k,l) + \chi_2(k,l) )\|\de_{y}\theta_k\|\|\partial_y \phi_{k-l}\|_{L^\infty}|l|\|y\theta_{l}\|\n\\ 
		&+{\rm e}^{2\delta_0 \kappa^{\fr12}t}\sum_{\substack{ k,l \neq 0 \\ k \neq l }}|k|^{2m}( \chi_1(k,l) + \chi_2(k,l) )\|y\de_{y}\theta_k\||k-l|\|\phi_{k-l}\|_{L^\infty} \|\partial_{y} \theta_{l}\|\n\\
		\lesssim&{\rm e}^{2\delta_0 \kappa^{\fr12}t}\sum_{\substack{ k,l \neq 0 \\ k \neq l }}(|k|^{m+\fr34}\|\de_{y}\theta_k\|)(|l|^{m+\fr14}\|y\theta_{l}\|)\|\partial_y \phi_{k-l}\|_{L^\infty}\n\\ 
		&+{\rm e}^{2\delta_0 \kappa^{\fr12}t}\sum_{\substack{ k,l \neq 0 \\ k \neq l }}(|k|^{m}\|\de_{y}\theta_k\|)(|l|^{\fr12}\|y\theta_{l}\|)(|k-l|^{m+\fr12}\|\partial_y \phi_{k-l}\|_{L^\infty})\n\\
		&+{\rm e}^{2\delta_0 \kappa^{\fr12}t}\sum_{\substack{ k,l \neq 0 \\ k \neq l }}(|k|^{m+\fr14}\|y\de_{y}\theta_k\|)(|l|^{m-\fr14}\|\partial_y \theta_{l}\|)\||k-l|\phi_{k-l}\|_{L^\infty}\n\\ 
		&+{\rm e}^{2\delta_0 \kappa^{\fr12}t}\sum_{\substack{ k,l \neq 0 \\ k \neq l }}(|k|^{m}\|y\de_{y}\theta_k\|)\|\partial_y \theta_{l}\|(|k-l|^{m+\fr34}\|\phi_{k-l}\|_{L^\infty})\n\\
		\lesssim&\mu^{-\fr12}\mathcal{E}_{\theta,\neq}^{\fr12}\mathcal{D}_{\theta,\neq}^{\fr12}\left(\sum_{\substack{ k \neq 0 }}|k|^{2m}\|\partial_y \phi_{k}\|\|\partial_{yy} \phi_{k}\|\right)^{\fr12}+\mu^{-\fr14}\mathcal{E}_{\theta,\neq}^{\fr12}\mathcal{D}_{\theta,\neq}^{\fr12}\left(\sum_{\substack{ k \neq 0 }}|k|^{2m}\||k|\partial_y \phi_{k}\|\|\partial_{yy} \phi_{k}\|\right)^{\fr12}\n\\
		&+\mathcal{E}_{\theta,\neq}^{\fr12}\mathcal{D}_{\theta,\neq}^{\fr12}\left(\sum_{\substack{ k \neq 0 }}|k|^{2m+2}\|\partial_y \phi_{k}\|\|\partial_{yy} \phi_{k}\|\right)^{\fr12}+\mu^{-\fr34}\mathcal{D}_{\theta,\neq}\left(\sum_{\substack{ k \neq 0 }}|k|^{2m+\fr12}\||k|\partial_y \phi_{k}\|\|\partial_{yy} \phi_{k}\|\right)^{\fr12}\n\\ 
		\lesssim&(\nu^{-\fr14}\mu^{-\fr12}+\nu^{-\fr12})\mathcal{E}_{\theta,\neq}^{\fr12}\mathcal{D}_{\theta,\neq}^{\fr12}\mathcal{D}_{\om,\neq}^{\fr12}+\nu^{\fr14}\mu^{-\fr34}\mathcal{E}_{\om,\neq}^{\fr12}\mathcal{D}_{\theta,\neq}. 
	\end{align}
Putting \eqref{i31}--\eqref{i32} into \eqref{I3} shows
\begin{align}\label{i3}
	I_3\lesssim& \mu^{-\frac{2}{3}}\|u_{0}\|_{L^\infty}\mathcal{D}_{\theta,\neq}+\nu^{\frac{1}{4}}\mu^{-\frac{3}{4}} \mathcal{E}^{\frac{1}{2}}_{\om,\neq} \mathcal{D}^{\frac{1}{2}}_{\theta, \neq}\mathcal{D}^{\frac{1}{2}}_{\theta, 0}+(\nu^{-\fr14}\mu^{-\fr12}+\nu^{-\fr12})\mathcal{E}_{\theta,\neq}^{\fr12}\mathcal{D}_{\theta,\neq}^{\fr12}\mathcal{D}_{\om,\neq}^{\fr12}\n\\ 
	&+\nu^{\fr14}\mu^{-\fr34}\mathcal{E}_{\om,\neq}^{\fr12}\mathcal{D}_{\theta,\neq}.
\end{align}
{\it \underline{Estimate of $I_4$:}} First of all, we decompose it into the following two parts
\begin{align}\label{I4}
I_4=&2^{-11}{\rm e}^{2\delta_0 \kappa^{\fr12}t}\sum_{k \neq 0}|k|^{2m-1}\RR\l ik\theta_k,-\sum_{l \in \Z} \partial_y \phi_{k-l} \cdot il  \theta_{l}+ \sum_{l \in \Z} i(k-l)\phi_{k-l} \cdot \partial_{y} \theta_{l}\r\n\\ 
=&2^{-11}{\rm e}^{2\delta_0 \kappa^{\fr12}t}\sum_{k \neq 0}|k|^{2m-1}\RR\l ik\theta_k, -\partial_{y} \phi_{0} \cdot ik \theta_k +ik \phi_{k} \cdot  \partial_{y} \theta_0\rangle\n\\ 
&+2^{-11}{\rm e}^{2\delta_0 \kappa^{\fr12}t}\sum_{k \neq 0}|k|^{2m-1}\RR\l ik\theta_k,  \nabla^\perp_{k-l} \phi_{k-l} \cdot \nabla_l \theta_l \rangle:= I_4^1+I_4^2.			
\end{align}
Recalling the definitions \eqref{2.2}--\eqref{2.13} and using H\"older inequality give
\begin{align}\label{i41}
			I_4^1\lesssim&{\rm e}^{2\delta_0 \kappa^{\fr12}t}\sum_{k \neq 0}|k|^{2m+1}\|\theta_k\|^2\|\partial_y \phi_{0}\|_{L^\infty}+{\rm e}^{2\delta_0 \kappa^{\fr12}t}\sum_{k \neq 0}|k|^{2m+1}\|\theta_k\|\|\phi_k\|_{L^\infty}\| \partial_{y} \theta_0\|\n\\
			\lesssim&{\rm e}^{2\delta_0 \kappa^{\fr12}t}\sum_{k \neq 0}(|k|^{2m+\fr12}\|\theta_k\|^2)^{\fr23}(|k|^{2m+2}\|\theta_k\|^2)^{\fr13}\|\partial_y \phi_{0}\|_{L^\infty}\n\\ 
			&+{\rm e}^{2\delta_0 \kappa^{\fr12}t}\sum_{k \neq 0}|k|^{m}\|\theta_k\||k|^{m+1}\|\phi_k\|_{L^\infty}\| \partial_{y} \theta_0\|\n\\
\lesssim&\mu^{-\frac{2}{3}}\|u_{0}\|_{L^\infty}\mathcal{D}_{\theta,\neq}+\mu^{-\frac{1}{4}}{\rm e}^{\delta_0 \kappa^{\fr12}t}\mathcal{D}^{\frac{1}{2}}_{\theta, \neq}\left(\sum_{k \neq 0}|k|^{2m+2}\|\phi_k\|_{L^\infty}^2\right)^{\fr12}\|\partial_{y} \theta_0\|\n\\
\lesssim& \mu^{-\frac{2}{3}}\|u_{0}\|_{L^\infty}\mathcal{D}_{\theta,\neq}+\mu^{-\frac{1}{2}} \mathcal{E}^{\frac{1}{2}}_{\theta, 0} \mathcal{D}^{\frac{1}{2}}_{\theta, \neq}\mathcal{D}^{\frac{1}{2}}_{\om, \neq}.
		\end{align}
    For $I_4^2$, we first use the two characteristic functions  to
 decompose it into the following two parts, and then by \eqref{psi} and $m>\fr12$, we have
	\begin{align}\label{i42}
		I_4^2\lesssim&{\rm e}^{2\delta_0 \kappa^{\fr12}t}\sum_{\substack{ k,l \neq 0 \\ k \neq l }}|k|^{2m}( \chi_1(k,l) + \chi_2(k,l) )\|\theta_k\|\|\partial_y \phi_{k-l}\|_{L^\infty}|l|\|\theta_{l}\|\n\\ 
		&+{\rm e}^{2\delta_0 \kappa^{\fr12}t}\sum_{\substack{ k,l \neq 0 \\ k \neq l }}|k|^{2m}( \chi_1(k,l) + \chi_2(k,l) )\|\theta_k\||k-l|\|\phi_{k-l}\|_{L^\infty} \|\partial_{y} \theta_{l}\|\n\\
		\lesssim&{\rm e}^{2\delta_0 \kappa^{\fr12}t}\sum_{\substack{ k,l \neq 0 \\ k \neq l }}(|k|^{m}\|\theta_k\|)(|l|^{m+1}\|\theta_{l}\|)\|\partial_y \phi_{k-l}\|_{L^\infty}\n\\ 
		&+{\rm e}^{2\delta_0 \kappa^{\fr12}t}\sum_{\substack{ k,l \neq 0 \\ k \neq l }}(|k|^{m}\|\theta_k\|)(|l|^{\fr14}\|\theta_{l}\|)(|k-l|^{m+\fr34}\|\partial_y \phi_{k-l}\|_{L^\infty})\n\\
		&+{\rm e}^{2\delta_0 \kappa^{\fr12}t}\sum_{\substack{ k,l \neq 0 \\ k \neq l }}(|k|^{m}\|\theta_k\|)(|l|^{m}\|\partial_y \theta_{l}\|)\||k-l|\phi_{k-l}\|_{L^\infty}\n\\ 
		&+{\rm e}^{2\delta_0 \kappa^{\fr12}t}\sum_{\substack{ k,l \neq 0 \\ k \neq l }}(|k|^{m}\|\theta_k\|)\|\partial_y \theta_{l}\|(|k-l|^{m+1}\|\phi_{k-l}\|_{L^\infty})\n\\
		\lesssim&\mu^{-\fr12}\mathcal{E}_{\theta,\neq}^{\fr12}\mathcal{D}_{\theta,\neq}^{\fr12}\left(\sum_{\substack{ k \neq 0 }}|k|^{2m}\|\partial_y \phi_{k}\|\|\partial_{yy} \phi_{k}\|\right)^{\fr12}+\mu^{-\fr14}\mathcal{E}_{\theta,\neq}^{\fr12}\mathcal{D}_{\theta,\neq}^{\fr12}\left(\sum_{\substack{ k \neq 0 }}|k|^{2m+\fr12}\||k|\partial_y \phi_{k}\|\|\partial_{yy} \phi_{k}\|\right)^{\fr12}\n\\
		&+\mu^{-\fr12}\mathcal{E}_{\theta,\neq}^{\fr12}\mathcal{D}_{\theta,\neq}^{\fr12}\left(\sum_{\substack{ k \neq 0 }}|k|^{2m+2}\|\partial_y \phi_{k}\|\|\partial_{yy} \phi_{k}\|\right)^{\fr12}\n\\ 
		\lesssim&(\nu^{-\fr14}\mu^{-\fr12}+\nu^{-\fr12}+\mu^{-\fr12})\mathcal{E}_{\theta,\neq}^{\fr12}\mathcal{D}_{\theta,\neq}^{\fr12}\mathcal{D}_{\om,\neq}^{\fr12}. 
	\end{align}
Substituting \eqref{i41}--\eqref{i42} into \eqref{I4} gives
\begin{align}\label{i4}
	I_4\lesssim& \mu^{-\frac{2}{3}}\|u_{0}\|_{L^\infty}\mathcal{D}_{\theta,\neq}+\mu^{-\frac{1}{2}} \mathcal{E}^{\frac{1}{2}}_{\theta, 0} \mathcal{D}^{\frac{1}{2}}_{\theta, \neq}\mathcal{D}^{\frac{1}{2}}_{\om, \neq}+(\nu^{-\fr14}\mu^{-\fr12}+\nu^{-\fr12}+\mu^{-\fr12})\mathcal{E}_{\theta,\neq}^{\fr12}\mathcal{D}_{\theta,\neq}^{\fr12}\mathcal{D}_{\om,\neq}^{\fr12}.
\end{align}
{\it \underline{Estimate of $I_5$:}} We rewrite this term as follows
\begin{align}\label{I5}
I_5=&2^{-9}\mu^{-\fr12}{\rm e}^{2\delta_0 \kappa^{\fr12}t}\sum_{k \neq 0}|k|^{2m+\fr12}\RR\l y^2\theta_k,\sum_{l \in \Z} \partial_y \phi_{k-l} \cdot il  \theta_{l}- \sum_{l \in \Z} i(k-l)\phi_{k-l} \cdot \partial_{y} \theta_{l}\r\n\\ 
=&2^{-9}\mu^{-\fr12}{\rm e}^{2\delta_0 \kappa^{\fr12}t}\sum_{k \neq 0}|k|^{2m+\fr12}\RR\l y^2\theta_k, \partial_{y} \phi_{0} \cdot ik \theta_k -ik \phi_{k} \cdot  \partial_{y} \theta_0\rangle\n\\ 
&-2^{-9}\mu^{-\fr12}{\rm e}^{2\delta_0 \kappa^{\fr12}t}\sum_{k \neq 0}|k|^{2m+\fr12}\RR\l y^2\theta_k,  \nabla^\perp_{k-l} \phi_{k-l} \cdot \nabla_l \theta_l \rangle:= I_5^1+I_5^2.		
\end{align}
For $k=l$ or $l=0$, we can get
\begin{align}\label{i51}
			I_5^1=&2^{-9}\mu^{-\fr12}{\rm e}^{2\delta_0 \kappa^{\fr12}t}\sum_{k \neq 0}|k|^{2m+\fr12}\RR\l y^2\theta_k, -ik \phi_{k} \cdot  \partial_{y} \theta_0\rangle\n\\
			\lesssim&\mu^{-\fr12}{\rm e}^{2\delta_0 \kappa^{\fr12}t}\sum_{k \neq 0}|k|^{m+\fr12}\|y\theta_k\||k|^{m+1}\|\phi_k\|_{L^\infty}\| y\partial_{y} \theta_0\|\n\\
			\lesssim&\mu^{-\frac{1}{2}}{\rm e}^{2\delta_0 \kappa^{\fr12}t}\mathcal{E}_{\theta,0}^{\fr12}\left(\sum_{k \neq 0}|k|^{2m+1}\|y\theta_k\|^2\right)^{\fr12}\left(\sum_{k \neq 0}|k|^{2m+2}\|\phi_k\|^2_{L^\infty}\right)^{\fr12}\n\\ 
			\lesssim& \mu^{-\frac{1}{2}}\mathcal{E}^{\frac{1}{2}}_{\theta,0} \mathcal{D}^{\frac{1}{2}}_{\theta, \neq}\mathcal{D}^{\frac{1}{2}}_{\om,\neq}.
		\end{align}
For $k\neq l$ and $l\neq 0$, 
\begin{align}\label{i52}
I_5^2=&2^{-9}\mu^{-\fr12}{\rm e}^{2\delta_0 \kappa^{\fr12}t}\sum_{\substack{ k,l \neq 0 \\ k \neq l }}|k|^{2m+\fr12}\chi_1(k,l)\RR\l y^2\theta_k,\partial_y \phi_{k-l} \cdot il  \theta_{l}\r\n\\
&+2^{-9}\mu^{-\fr12}{\rm e}^{2\delta_0 \kappa^{\fr12}t}\sum_{\substack{ k,l \neq 0 \\ k \neq l }}|k|^{2m+\fr12}\chi_2(k,l)\RR\l y^2\theta_k,\partial_y \phi_{k-l} \cdot il  \theta_{l}\r\n\\
&-2^{-9}\mu^{-\fr12}{\rm e}^{2\delta_0 \kappa^{\fr12}t}\sum_{\substack{ k,l \neq 0 \\ k \neq l }}|k|^{2m+\fr12}( \chi_1(k,l) + \chi_2(k,l) )\RR\l y^2\theta_k,i(k-l)\phi_{k-l} \cdot \partial_{y} \theta_{l}\r\n\\
:=&I_5^{2,1}+I_5^{2,2}+I_5^{2,3}.		
\end{align}
  By \eqref{psi} and $m>\fr12$, we have
	\begin{align}\label{i521}
		I_5^{2,1}\lesssim&\mu^{-\fr12}{\rm e}^{2\delta_0 \kappa^{\fr12}t}\sum_{\substack{ k,l \neq 0 \\ k \neq l }}(|k|^{m+\fr12}\|y\theta_k\|)(|l|^{m+1}\|\theta_{l}\|)\|y\partial_y \phi_{k-l}\|_{L^\infty}\n\\
		\lesssim&\mu^{-1}\mathcal{D}_{\theta,\neq}\left(\sum_{k\neq 0}|k|^{2m}\|y\partial_y \phi_{k}\|(\|\partial_y \phi_{k}\|+\|y\om_{k}\|)\right)^{\fr12}   \n\\ 
		\lesssim&\mu^{-1}\mathcal{D}_{\theta,\neq}\left(\sum_{k\neq 0}|k|^{2m}\|y\om_{k}\|(\|\partial_y \phi_{k}\|+\|y\om_{k}\|)\right)^{\fr12}   \n\\  
		\lesssim&\nu^{\fr14}\mu^{-1}\mathcal{E}_{\om,\neq}^{\fr12}\mathcal{D}_{\theta,\neq}, 
	\end{align}
\begin{align}\label{i522}
		I_5^{2,2}\lesssim&\mu^{-\fr12}{\rm e}^{2\delta_0 \kappa^{\fr12}t}\sum_{\substack{ k,l \neq 0 \\ k \neq l }}(|k|^{m+\fr12}\|y\theta_k\|)(|l|^{\fr14}\|y\theta_{l}\|)(|k-l|^{m+\fr34}\|\partial_y \phi_{k-l}\|_{L^\infty})\n\\
		\lesssim&\mu^{-\fr12}{\rm e}^{\delta_0 \kappa^{\fr12}t}\mathcal{D}_{\theta,\neq}^{\fr12}\left(\sum_{k\neq 0}|k|^{2m+\fr12}\||k|\partial_y \phi_{k}\|\|\partial_{yy} \phi_{k}\|\right)^{\fr12}\left(\sum_{k\neq 0}|k|^{2m+\fr12}\|y\theta_k\|^2\right)^{\fr12}   \n\\ 
		\lesssim&\nu^{-\fr14}\mu^{-\fr14}\mathcal{E}_{\theta,\neq}^{\fr12}\mathcal{D}_{\theta,\neq}^{\fr12}\mathcal{D}_{\om,\neq}^{\fr12}, 
	\end{align}
and
	\begin{align}\label{i523}
		I_5^{2,3}\lesssim&\mu^{-\fr12}{\rm e}^{2\delta_0 \kappa^{\fr12}t}\sum_{\substack{ k,l \neq 0 \\ k \neq l }}(|k|^{m+\fr14}\|y\theta_k\|)(|l|^{m+\fr14}\|y\partial_y \theta_{l}\|)|k-l|\|\phi_{k-l}\|_{L^\infty}\n\\ 
		&+\mu^{-\fr12}{\rm e}^{2\delta_0 \kappa^{\fr12}t}\sum_{\substack{ k,l \neq 0 \\ k \neq l }}(|k|^{m+\fr34}\|y\theta_k\|)(|k-l|^{m+\fr34}\|\phi_{k-l}\|_{L^\infty})\|y\de_y\theta_{l}\|\n\\
		\lesssim&\mu^{-\fr12}\mathcal{E}_{\theta,\neq}^{\fr12}\mathcal{D}_{\theta,\neq}^{\fr12}\left(\sum_{\substack{ k \neq 0 }}|k|^{2m+2}\|\phi_{k}\|_{L^\infty}^2\right)^{\fr12}+\mu^{-1}\mathcal{D}_{\theta,\neq}\left(\sum_{\substack{ k \neq 0 }}|k|^{2m+\fr12}\||k|\phi_{k}\|\|\partial_{y} \phi_{k}\|\right)^{\fr12}\n\\ 
		\lesssim&\mu^{-\fr12}\mathcal{E}_{\theta,\neq}^{\fr12}\mathcal{D}_{\theta,\neq}^{\fr12}\mathcal{D}_{\om,\neq}^{\fr12}+\nu^{\fr14}\mu^{-1}\mathcal{E}_{\om,\neq}^{\fr12}\mathcal{D}_{\theta,\neq}. 
	\end{align}
Putting \eqref{i51}--\eqref{i523} into \eqref{I5}, we have
\begin{align}\label{i5}
	I_5\lesssim&\mu^{-\frac{1}{2}}\mathcal{E}^{\frac{1}{2}}_{\theta,0} \mathcal{D}^{\frac{1}{2}}_{\theta, \neq}\mathcal{D}^{\frac{1}{2}}_{\om,\neq}+\nu^{\fr14}\mu^{-1}\mathcal{E}_{\om,\neq}^{\fr12}\mathcal{D}_{\theta,\neq}+(\nu^{-\fr14}\mu^{-\fr14}+\mu^{-\fr12})\mathcal{E}_{\theta,\neq}^{\fr12}\mathcal{D}_{\theta,\neq}^{\fr12}\mathcal{D}_{\om,\neq}^{\fr12}.
\end{align}
Combining \eqref{4.45}, \eqref{i1}, \eqref{i2}, \eqref{i3}, \eqref{i4} and \eqref{i5}, we obtain \eqref{2.17}.

\subsection{Estimate of $\omega$ }
\begin{lem}\label{l4.2}
	Let $(\omega_k,\theta_k)(t,x,y)$ solve \eqref{4.1}. Then we have the following balances:
 \begin{equation}\label{4.46}
  \frac12\ddt \|\omega_k\|^2 +\nu\|\nabla_k\omega_k\|^2=\RR\l\omega_k,ik\theta_k-(u\cdot \nabla \omega)_k\r,
  \end{equation}
  \begin{equation}\label{4.47}
  \frac12\ddt \|\de_y\omega_k\|^2 +\nu\|\de_y\nabla_k\omega_k\|^2+2\RR\l\de_y\omega_k, iky\omega_k\r=\RR\l\de_y\omega_k,ik\de_{y}\theta_k-\de_y(u\cdot \nabla \omega)_k\r,
  \end{equation}
  \begin{align}\label{4.48}
  &\ddt\RR\l iky \omega_k,\de_y\omega_k\r+2|k|^2\| y\omega_k\|^2+4|k|^2\|\de_{y}\phi_k\|^2 \n\\
  =&-2\nu\RR\l iky\de_{y}\omega_k, \Delta_k \omega_k\r-2\RR\l iky\de_{y}\omega_k, ik\theta_k-(u\cdot \nabla \omega)_k\r-\RR\l ik\omega_k, ik\theta_k-(u\cdot \nabla \omega)_k\r,
  \end{align}
  \begin{equation}\label{4.49}
  \frac12\ddt\| y\omega_k\|^2+\nu\|y\nabla_k\omega_k\|^2=\nu\|\omega_k\|^2-4\RR\l y\de_{y} \phi_k, ik\phi_k\r+\RR\l y^2\omega_k, ik\theta_k-(u\cdot \nabla \omega)_k\r,
  \end{equation}
  \begin{equation}\label{4.50}
  \frac12\ddt\| \nabla_k \phi_k\|^2 +\nu\|\omega_k\|^2=2\RR\l y\de_{y} \phi_k,ik\phi_k\r-\RR\l \phi_k, ik\theta_k-(u\cdot \nabla \omega)_k\r.
 \end{equation}
\end{lem}
\begin{proof}
Taking the inner product of system \eqref{4.1}$_1$ with $\omega_k$ and $-\de_{yy}\omega_k$ in $L^2$ respectively, and integrating the resulting equation by parts, we can obtain \eqref{4.46}--\eqref{4.47}. 

Turning to \eqref{4.48}, we use \eqref{4.1}$_1$ to compute
\begin{align}
\ddt\RR\l iky\omega_k,\de_y\omega_k\r=&\RR\l iky\de_{t} \omega_k,\de_y\omega_k\r+\RR\l iky\omega_k,\de_{yt}\omega_k\r=-2\RR\l iky\de_{y}\omega_k,\de_{t} \omega_k\r-\RR\l ik\omega_k,\de_{t}\omega_k\r\n\\ 
=&\RR\l 2iky\de_{y}\omega_k+ik\omega_k, iky^2\omega_k -2ik\phi_k-\nu \Delta_k\omega_k+(u\cdot \nabla \omega)_k-ik\theta_k\r   \n\\
=&-\nu\RR\l 2iky\de_{y}\omega_k+ik\omega_k, \Delta_k \omega_k\r+\RR\l 2iky\de_{y}\omega_k+ik\omega_k,(u\cdot \nabla \omega)_k-ik\theta_k\r   \n\\
&+\RR\l 2iky\de_{y}\omega_k+ik\omega_k, iky^2\omega_k -2ik\phi_k\r .
\end{align}
For the second term, we can compute
\begin{align}
\RR\l 2iky\de_{y}\omega_k+ik\omega_k, iky^2\omega_k\r=2\RR\l iky^3\de_{y}\omega_k, ik\omega_k\r+\RR\l ik\omega_k, iky^2\omega_k\r=-2|k|^2\| y\omega_k\|^2.\notag	
\end{align}
Lastly, the third term can be estimate as
\begin{align}
-2\RR\l 2iky\de_{y}\omega_k+ik\omega_k, ik\phi_k\r=&-4|k|^2\RR\l y\de_{y}\omega_k, \phi_k\r-2|k|^2\RR\l \omega_k, \phi_k\r	\n\\ 
=&-2|k|^2\RR\l y\de_{y}\omega_k, \phi_k\r+2|k|^2\RR\l y\omega_k, \de_{y}\phi_k\r\n\\ 
=&-2|k|^2\RR\l \de_y\Delta_k\phi_k,y\phi_k\r+2|k|^2\RR\l y\Delta_k\phi_k,\de_{y}\phi_k\r\n \\  
=&2|k|^2\RR\l y\Delta_k\phi_k,\de_{y}\phi_k\r-2|k|^2\RR\l \Delta_k(y\phi_k),\de_{y}\phi_k\r=-4|k|^2\| \de_{y}\phi_k\|^2,\n
\end{align}
and then \eqref{4.48} follows. For \eqref{4.49}, we have
\begin{align}
\frac12\ddt\| y\omega_k\|^2
&=\RR\l y\omega_k, y(2ik\phi_k+\nu\Delta_k \omega_k-iky^2\omega_k +ik\theta_k-(u\cdot \nabla \omega)_k)\r\notag\\
&=2\l y\Delta\phi_k,iky\phi_k\r+\nu\l y\omega_k, y \Delta_k \omega_k\r +\RR\l y\omega_k, y(ik\theta_k-(u\cdot \nabla \omega)_k)\r \notag\\
&=\nu\|\omega_k\|^2-\nu\|y\nabla_k\omega_k\|^2-4\RR\l y\de_{y}\phi_k,ik\phi_k\r+\RR\l y^2\omega_k,ik\theta_k-(u\cdot \nabla \omega)_k)\r.\n
 \end{align}
For \eqref{4.50}, we multiply \eqref{4.1}$_1$ by $-\phi_k$ and integrate it to obtain
\begin{align}
-\l\de_t\omega_k,\phi_k\r -\l iky^2\omega_k,\phi_k\r =-\l\nu\Delta_k\omega_k+ik\theta_k-(u\cdot \nabla \omega)_k,\phi_k\r.\n
\end{align}
Notice that
\begin{align}
 -\l iky^2\Delta_k \phi_k,\phi_k\r=2\l iky\de_{y} \phi_k,\phi_k\r+\l iky^2\nabla_k \phi_k,\nabla_k\phi_k\r=2\l iky \de_{y} \phi_k,\phi_k\r,\n
\end{align}
hence we have \eqref{4.50}.
The proof of Lemma \ref{l4.2} is over.
\end{proof}
{\bf Proof of \eqref{2.16}:} We firstly observe that 
\begin{equation}\n
 \l \de_{y}\phi_k, y\omega_k \r=\l \de_{y}\phi_k, y\de_{yy}\phi_k \r-\l \de_{y}\phi_k, y|k|^2\phi_k \r=-\frac12\|\de_{y}\phi_k\|^2+\frac12|k|^2\|\phi_k\|^2,
\end{equation}
then it follows that
\begin{align}\label{4.59}
|k|^2\|\phi_k\|^2=\|\de_{y}\phi_k\|^2+2\l\de_{y}\phi_k, y\omega_k \r\leq 2\|\de_{y}\phi_k\|^2+\|y\omega_k\|^2.
\end{align}
According to the energy estimation in Lemma \ref{l4.2}, \eqref{4.59} and the definition of $E_{\omega}[\omega_k]$ in \eqref{2.6}, one can verify that the following inequality holds
\begin{align}\label{4.60}
&\ddt E_{\omega}[\omega_k] +\nu\|\nabla_k\omega_k\|^2+\nu^{\frac32}|k|^{-\frac12}\|\de_y\nabla_k\omega_k\|^2+4|k|\| y\omega_k\|^2+|k|\|\nabla_k\phi_k\|^2\n\\ 
&\qquad+100\nu^{\frac12}|k|^{\frac12}\left(\|y\nabla_k\omega_k\|^2 +\|\omega_k\|^2\right)\n\\
\leq&-2\nu^{\frac12}|k|^{-\frac12}\RR\l iky\omega_k,\de_y\omega_k\r-5\nu|k|^{-1}\RR\l \Delta_k \omega_k, iky\de_{y}\omega_k\r+\RR\l\omega_k,ik\theta_k-(u\cdot \nabla \omega)_k\r\n\\ 
&+\nu^{\frac12}|k|^{-\frac12}\RR\l\de_y\omega_k,ik\de_{y}\theta_k-\de_y(u\cdot \nabla \omega)_k\r-5|k|^{-1}\RR\l iky\de_{y}\omega_k, ik\theta_k-(u\cdot \nabla \omega)_k\r\n\\ 
&-\fr52|k|^{-1}\RR\l ik\omega_k, ik\theta_k-(u\cdot \nabla \omega)_k\r+100\nu^{-\frac12}|k|^{\frac12}\RR\l y^2\omega_k, ik\theta_k-(u\cdot \nabla \omega)_k\r\n\\ 
&-200\nu^{-\frac12}|k|^{\frac12}\RR\l \phi_k, ik\theta_k-(u\cdot \nabla \omega)_k\r.
\end{align}

We first handle the first two terms on the right-hand side. Using Young inequality, we obtain
\begin{align}\label{4.61}
-2\nu^{\frac12}|k|^{-\frac12}\RR\l iky\omega_k,\de_y\omega_k\r\leq \frac{\nu}{2} \|\nabla_k \omega_k\|^2+ 2|k|\|y\omega_k\|^2,
\end{align}
\begin{align}\label{4.62}
-5\nu|k|^{-1}\RR\l \Delta_k \omega_k, iky\de_{y}\omega_k\r=&5\nu|k|^{-1}\RR\l y|k|^2 \omega_k, ik\de_{y}\omega_k\r-5\nu|k|^{-1}\RR\l \de_{yy} \omega_k, iky\de_{y}\omega_k\r\n\\ 
\leq& 5\nu|k|^{2}\|y\omega_k\|\|\de_{y}\omega_k\|+5\nu\|\de_{yy} \omega_k\|\|y\de_{y}\omega_k\| \n\\ 
\leq&10\nu\|\de_{y} \nabla_k\omega_k\|\|y\nabla_k\omega_k\|  \n\\
\leq&\frac{1}{2}\nu^{\fr32}|k|^{-\fr12}\|\de_{y} \nabla_k\omega_k\|^2  +50\nu^{\fr12}|k|^{\fr12}\|y\nabla_k\omega_k\|^2.
\end{align}
Then, substituting \eqref{4.61}--\eqref{4.62} into \eqref{4.60}, and multiplying both sides of the resulting equation by ${\rm e}^{2\delta_0 \kappa^{\fr12}t}$, after performing some integration by parts, we arrive at
\begin{align}\label{4.63}
&\ddt \mathcal{E}_{\om,\neq} +\fr12\mathcal{D}_{\om,\neq}-2\delta_0\kappa^{\fr12 }\mathcal{E}_{\om,\neq}\n\\
\leq&{\rm e}^{2\delta_0 \kappa^{\fr12}t}\sum_{k \neq 0}|k|^{2m}\RR\l\omega_k,ik\theta_k\r+\nu^{\frac12}{\rm e}^{2\delta_0 \kappa^{\fr12}t}\sum_{k \neq 0}|k|^{2m-\frac12}\RR\l\de_y\omega_k,ik\de_{y}\theta_k\r\n\\ 
&-5{\rm e}^{2\delta_0 \kappa^{\fr12}t}\sum_{k \neq 0}|k|^{2m-1}\RR\l iky\de_{y}\omega_k, ik\theta_k\r-\fr52{\rm e}^{2\delta_0 \kappa^{\fr12}t}\sum_{k \neq 0}|k|^{2m-1}\RR\l ik\omega_k, ik\theta_k\r\n\\ 
&+100\nu^{-\frac12}{\rm e}^{2\delta_0 \kappa^{\fr12}t}\sum_{k \neq 0}|k|^{2m+\fr12}\RR\l y^2\omega_k, ik\theta_k\r-200\nu^{-\frac12}{\rm e}^{2\delta_0 \kappa^{\fr12}t}\sum_{k \neq 0}|k|^{2m+\fr12}\RR\l \phi_k, ik\theta_k\r\n\\ 
&-{\rm e}^{2\delta_0 \kappa^{\fr12}t}\sum_{k \neq 0}|k|^{2m}\RR\l\omega_k,(u\cdot \nabla \omega)_k\r-\nu^{\frac12}{\rm e}^{2\delta_0 \kappa^{\fr12}t}\sum_{k \neq 0}|k|^{2m-\frac12}\RR\l\de_y\omega_k,\de_y(u\cdot \nabla \omega)_k\r\n\\ 
&+5{\rm e}^{2\delta_0 \kappa^{\fr12}t}\sum_{k \neq 0}|k|^{2m-1}\RR\l iky\de_{y}\omega_k, (u\cdot \nabla \omega)_k\r+\fr52{\rm e}^{2\delta_0 \kappa^{\fr12}t}\sum_{k \neq 0}|k|^{2m-1}\RR\l ik\omega_k, (u\cdot \nabla \omega)_k\r\n\\ 
&-100\nu^{-\frac12}{\rm e}^{2\delta_0 \kappa^{\fr12}t}\sum_{k \neq 0}|k|^{2m+\fr12}\RR\l y^2\omega_k, (u\cdot \nabla \omega)_k\r\n\\ 
&+200\nu^{-\frac12}{\rm e}^{2\delta_0 \kappa^{\fr12}t}\sum_{k \neq 0}|k|^{2m+\fr12}\RR\l \phi_k, (u\cdot \nabla \omega)_k\r:=\sum_{i=1}^{12}J_i.
\end{align}
First of all, it is easy to check that if $\delta_0\leq 2^{-16}$, 
\begin{equation}\label{4.64}
	2\delta_0\kappa^{\fr12 }\mathcal{E}_{\om,\neq}\leq \frac{1}{16}\mathcal{D}_{\om,\neq}.
\end{equation}
Next, using H\"older inequality, Young inequality and definitions \eqref{2.2}--\eqref{2.13} yields
\begin{align}\label{4.65}
	\sum_{i=1}^{4}J_i\leq&\fr72{\rm e}^{2\delta_0 \kappa^{\fr12}t}\sum_{k \neq 0}|k|^{2m+1}\|\omega_k\|\|\theta_k\|+\nu^{\frac12}{\rm e}^{2\delta_0 \kappa^{\fr12}t}\sum_{k \neq 0}|k|^{2m+\frac12}\|\de_{yy}\omega_k\|\|\theta_k\|\n\\ 
&+5{\rm e}^{2\delta_0 \kappa^{\fr12}t}\sum_{k \neq 0}|k|^{2m+1}\|y\de_{y}\omega_k\|\|\theta_k\|\n\\
\leq&\fr72\nu^{-\frac14}{\rm e}^{2\delta_0 \kappa^{\fr12}t}\sum_{k \neq 0}(\nu^{\frac14}|k|^{m+\fr14}\|\omega_k\|)(|k|^{m+\fr34}\|\theta_k\|)\n\\ 
&+\nu^{-\frac14}{\rm e}^{2\delta_0 \kappa^{\fr12}t}\sum_{k \neq 0}(\nu^{\frac34}|k|^{m-\fr14}\|\de_{yy}\omega_k\|)(|k|^{m+\frac34}\|\theta_k\|)\n\\ 
&+5\nu^{-\frac14}{\rm e}^{2\delta_0 \kappa^{\fr12}t}\sum_{k \neq 0}(\nu^{\frac14}|k|^{m+\fr14}\|y\de_{y}\omega_k\|)(|k|^{m+\frac34}\|\theta_k\|)\n\\
\leq& \frac{1}{16}\mathcal{D}_{\om,\neq}+C_1\nu^{-\frac12}{\rm e}^{2\delta_0 \kappa^{\fr12}t}\sum_{k \neq 0}(|k|^{2m+\frac12}\|\theta_k\|^2)^{\fr13}(|k|^{2m+2}\|\theta_k\|^2)^{\fr23}\n\\
\leq& \frac{1}{16}\mathcal{D}_{\om,\neq}+C_1\nu^{-\frac12}\mu^{-\frac56}\mathcal{D}_{\theta,\neq},
\end{align}
\begin{align}\label{4.66}
J_5+J_6\leq& 100\nu^{-\frac12}{\rm e}^{2\delta_0 \kappa^{\fr12}t}\sum_{k \neq 0}|k|^{2m+\fr32}\|y\omega_k\|\|y\theta_k\|+200\nu^{-\frac12}{\rm e}^{2\delta_0 \kappa^{\fr12}t}\sum_{k \neq 0}|k|^{2m+\fr32}\|\phi_k\|\|\theta_k\|\n\\ 
\leq&100\nu^{-\frac12}{\rm e}^{2\delta_0 \kappa^{\fr12}t}\sum_{k \neq 0}(|k|^{m+\fr12}\|y\omega_k\|)(|k|^{m+1}\|y\theta_k\|)\n\\ 
&+200\nu^{-\frac12}{\rm e}^{2\delta_0 \kappa^{\fr12}t}\sum_{k \neq 0}(|k|^{m+\fr32}\|\phi_k\|)(|k|^{m}\|\theta_k\|)\n\\
\leq& \frac{1}{16}\mathcal{D}_{\om,\neq}+C_2\nu^{-1}{\rm e}^{2\delta_0 \kappa^{\fr12}t}\left(\sum_{k \neq 0}(|k|^{2m+1}\|y\theta_k\|^2)^{\fr13}(|k|^{2m+\fr52}\|y\theta_k\|^2)^{\fr23}+\sum_{k \neq 0}|k|^{2m}\|\theta_k\|^2\right)\n\\
\leq& \frac{1}{16}\mathcal{D}_{\om,\neq}+C_2\nu^{-1}\mu^{-\fr12}\mathcal{D}_{\theta,\neq}.	
\end{align}

Similar to \eqref{i1}, \eqref{i2}, \eqref{i3}, \eqref{i4} and \eqref{i5}, we can estimate $J_7$--$J_{11}$ as
\begin{align}\label{j7}
       J_7\lesssim\nu^{-\frac{1}{2}} \mathcal{E}^{\frac{1}{2}}_{\omega,\neq} \mathcal{D}^{\frac{1}{2}}_{\omega, \neq}\mathcal{D}^{\frac{1}{2}}_{\omega, 0}+\nu^{-\frac{2}{3}}\mathcal{E}^{\frac{1}{2}}_{\omega,\neq} \mathcal{D}_{\omega,\neq},
\end{align}
\begin{align}\label{j8}
       J_8\lesssim& \nu^{-\frac12}\mathcal{E}_{\omega,0}^{\fr12}\mathcal{D}_{\omega,\neq}+\nu^{-\frac{13}{24}}\mathcal{E}_{\om,\neq}^{\frac12}\mathcal{D}_{\omega,\neq},
\end{align}
\begin{align}\label{j9}
       J_9\lesssim& \nu^{-\frac{2}{3}}\|u_{0}\|_{L^\infty}\mathcal{D}_{\omega,\neq}+\nu^{\frac{1}{4}}\nu^{-\frac{3}{4}} \mathcal{E}^{\frac{1}{2}}_{\om,\neq} \mathcal{D}^{\frac{1}{2}}_{\omega, \neq}\mathcal{D}^{\frac{1}{2}}_{\omega, 0}+\nu^{-\fr34}\mathcal{E}_{\omega,\neq}^{\fr12}\mathcal{D}_{\omega,\neq},
\end{align}
\begin{align}\label{j10}
      J_{10}\lesssim& \nu^{-\frac{2}{3}}\|u_{0}\|_{L^\infty}\mathcal{D}_{\omega,\neq}+\nu^{-\frac{1}{2}} \mathcal{E}^{\frac{1}{2}}_{\omega, 0} \mathcal{D}_{\omega, \neq}+\nu^{-\fr34}\mathcal{E}_{\omega,\neq}^{\fr12}\mathcal{D}_{\omega,\neq},
\end{align}
\begin{align}\label{j11}
       J_{11}\lesssim&\nu^{-\frac{1}{2}}\mathcal{E}^{\frac{1}{2}}_{\omega,0} \mathcal{D}_{\omega, \neq}+\nu^{-\fr34}\mathcal{E}_{\om,\neq}^{\fr12}\mathcal{D}_{\omega,\neq}.
\end{align}
For $J_{12}$, we can prove 
\begin{align}\label{j12}
J_{12}=&200\nu^{-\frac{1}{2}}{\rm e}^{2\delta_0 \kappa^{\fr12}t}\sum_{k \neq 0}|k|^{2m+\fr12}\RR\langle \phi_k, -\sum_{l \in \Z} \partial_y \phi_{k-l} \cdot il  \omega_{l}+ \sum_{l \in \Z} i(k-l)\phi_{k-l} \cdot \partial_{y} \omega_{l}  \rangle\n\\
=&-200\nu^{-\frac{1}{2}}{\rm e}^{2\delta_0 \kappa^{\fr12}t}\sum_{k \neq 0}|k|^{2m+\fr12} \RR\langle \phi_k, \partial_{y} \phi_{0} \cdot ik \omega_k \rangle \n\\ 
&+200\nu^{-\frac{1}{2}}{\rm e}^{2\delta_0 \kappa^{\fr12}t}\sum_{k \neq 0}|k|^{2m+\fr12} \RR\langle \phi_k,  ik \phi_{k} \cdot  \partial_{y} \omega_0 \rangle\n\\
&+200\nu^{-\frac{1}{2}}{\rm e}^{2\delta_0 \kappa^{\fr12}t}\sum_{\substack{ k,l \neq 0 \\ k \neq l }} |k|^{2m+\fr12} \RR\langle \phi_k,  -\partial_y \phi_{k-l} \cdot il  \omega_{l}+ i(k-l)\phi_{k-l} \cdot \partial_{y} \omega_{l} \rangle\n\\
\lesssim&\nu^{-\frac{1}{2}}{\rm e}^{2\delta_0 \kappa^{\fr12}t}\sum_{k \neq 0}(|k|^{m+\fr32}\|\phi_k\|)(|k|^{m}\omega_k \|)\|\partial_{y} \phi_{0}\|_{L^\infty}\n\\ 
&+\nu^{-\frac{1}{2}}{\rm e}^{2\delta_0 \kappa^{\fr12}t}\sum_{\substack{ k,l \neq 0 \\ k \neq l }}(|k|^{m+\fr32}\|\phi_k\|)(\||l|^{m}\omega_l \|)\|\partial_{y} \phi_{k-l}\|_{L^\infty}\n\\ 
&+\nu^{-\frac{1}{2}}{\rm e}^{2\delta_0 \kappa^{\fr12}t}\sum_{\substack{ k,l \neq 0 \\ k \neq l }}(|k|^{m+\fr34}\|\phi_k\|)\|\omega_l \|(|k-l|^{m+\fr34}\|\partial_{y} \phi_{k-l}\|_{L^\infty})\n\\ 
\lesssim&\left(\nu^{-\fr34}\|u_{0}\|_{L^\infty}\mathcal{D}_{\omega,\neq}^{\fr12}+\nu^{-\frac{1}{2}}{\rm e}^{\delta_0 \kappa^{\fr12}t}\mathcal{E}_{\omega,\neq}^{\fr12}\left(\sum_{\substack{ k \neq 0 }}|k|^{2m}\|\partial_{y} \phi_{k}\|\|\partial_{yy} \phi_{k}\|\right)^{\fr12}\right)\left(\sum_{\substack{ k \neq 0 }}|k|^{2m+3}\|\phi_{k}\|^2\right)^{\fr12}\n\\ 
&+\nu^{-\frac{1}{2}}{\rm e}^{\delta_0 \kappa^{\fr12}t}\mathcal{E}_{\omega,\neq}^{\fr12}\left(\sum_{\substack{ k \neq 0 }}|k|^{2m+\fr32}\|\phi_{k}\|^2\right)^{\fr12}\left(\sum_{\substack{ k \neq 0 }}|k|^{2m+\fr32}\|\partial_{y} \phi_{k}\|\|\partial_{yy} \phi_{k}\|\right)^{\fr12}\n\\ 
 \lesssim&\nu^{-\fr34}\|u_{0}\|_{L^\infty}\mathcal{D}_{\omega,\neq}+\nu^{-\fr34}\mathcal{E}_{\om,\neq}^{\fr12}\mathcal{D}_{\omega,\neq}.      
\end{align}	
Putting \eqref{4.64}--\eqref{j12} into \eqref{4.63} gives \eqref{2.16}.

Thus,we have completed the proof of Propositions \ref{p2} and \ref{p4}.

\vspace{1cm}

\noindent{\bf Conflict of Interest} The authors declare that they have no conflict of interest.

\vspace{0.5cm}
\noindent{\bf Acknowledgments:} This research was not supported by any funding.

\end{document}